\documentclass[namedate, webpdf, imanum]{ima-authoring-template}

\graphicspath{{Fig/}}

\theoremstyle{thmstyletwo}%
\newtheorem{theorem}{Theorem}
\theoremstyle{remark}
\newtheorem{remark}{Remark}%
\theoremstyle{thmstyletwo}
\newtheorem{definition}{Definition}
\newtheorem{lemma}{Lemma}
\numberwithin{equation}{section}
\theoremstyle{definition}
\newtheorem{assumption}{Assumption}[section]

\usepackage{graphicx, lipsum}

\usepackage{tikz-cd}
\usepackage{mathrsfs}
\usepackage{subcaption}
\usepackage{chemformula}
\usepackage{amssymb}
\usepackage{mathtools}
\usepackage{natbib}
\usepackage{url}
\usepackage{nicematrix}
\makeatletter
    \newcommand{\nicemat}{\begin{bNiceArray}{ccc|cc}[margin]}
    \newcommand{\nicediscretemat}{\begin{bNiceArray}{ccc|cc}[left-margin=.6em,right-margin=1.15em]}
    \newcommand{\nicevec}{\begin{bNiceArray}{c}[left-margin=.4em,right-margin=.4em]}
\newcommand{\dx}{\mathrm{d}x}
\newcommand{\ds}{\mathrm{d}s}
\newcommand{\bn}{\mathbf{n}}
\DeclareMathOperator{\Div}{div}
\DeclareMathOperator{\tr}{tr}
\DeclareMathOperator{\curl}{curl}
\DeclareMathOperator{\dev}{dev}
\DeclareMathOperator{\sph}{sph}
\newcommand{\hdivs}{H(\mathrm{div};\mathbb{S})}
\usepackage{cleveref}

\crefformat{assumption}{assumption~#2#1#3}
\Crefformat{assumption}{Assumption~#2#1#3}
\crefformat{equation}{#2(#1)#3}
\Crefformat{equation}{#2Equation~(#1)#3}

\def\Xint#1{\mathchoice 
{\XXint\displaystyle\textstyle{#1}}
{\XXint\textstyle\scriptstyle{#1}}
{\XXint\scriptstyle\scriptscriptstyle{#1}}
{\XXint\scriptscriptstyle\scriptscriptstyle{#1}}
\!\int}
\def\XXint#1#2#3{{\setbox0=\hbox{$#1{#2#3}{\int}$ }
\vcenter{\hbox{$#2#3$ }}\kern-.6\wd0}}

\def\dashint{\Xint-}

\DeclareFontFamily{U}{matha}{\hyphenchar\font45}
\DeclareFontShape{U}{matha}{m}{n}{
      <5> <6> <7> <8> <9> <10> gen * matha
      <10.95> matha10 <12> <14.4> <17.28> <20.74> <24.88> matha12
      }{}
\DeclareSymbolFont{matha}{U}{matha}{m}{n}
\DeclareFontFamily{U}{mathx}{\hyphenchar\font45}
\DeclareFontShape{U}{mathx}{m}{n}{
      <5> <6> <7> <8> <9> <10>
      <10.95> <12> <14.4> <17.28> <20.74> <24.88>
      mathx10
      }{}
\DeclareSymbolFont{mathx}{U}{mathx}{m}{n}
\DeclareMathSymbol{\obot}{2}{matha}{"6B}
\DeclareMathSymbol{\bigobot}{1}{mathx}{"CB}

\begin{document}

\DOI{}
\copyrightyear{2022}
\vol{00}
\pubyear{2022}
\access{}
\appnotes{Paper}
\copyrightstatement{}
\firstpage{1}

\title[FEM for multicomponent flow]{Finite element methods for multicomponent convection-diffusion}

\author{Francis R.~A.~Aznaran\ORCID{0000-0001-9843-9852}
\address{\orgdiv{Mathematical Institute}, \orgname{University of Oxford}, \orgaddress{\street{Oxford}, \postcode{OX2 6GG}, \country{UK}}}}
\author{Patrick E.~Farrell*\ORCID{0000-0002-1241-7060}
\address{\orgdiv{Mathematical Institute}, \orgname{University of Oxford}, \orgaddress{\street{Oxford}, \postcode{OX2 6GG}, \country{UK}}}}
\author{Charles W.~Monroe\ORCID{0000-0002-9894-5023}
\address{\orgdiv{Department of Engineering Science}, \orgname{University of Oxford}, \orgaddress{\street{Oxford}, \postcode{OX1 3PJ}, \country{UK;}}}\vspace{-2.5mm}
\address{\orgdiv{The Faraday Institution}, \orgname{Harwell Campus}, \orgaddress{\street{Didcot}, \postcode{OX11 ORA}, \country{UK}}}
}
\author{Alexander J.~Van-Brunt\ORCID{0000-0002-7259-7489}
\address{\orgdiv{Mathematical Institute}, \orgname{University of Oxford}, \orgaddress{\street{Oxford}, \postcode{OX2 6GG}, \country{UK;}}}\vspace{-2.5mm}
\address{\orgdiv{The Faraday Institution}, \orgname{Harwell Campus}, \orgaddress{\street{Didcot}, \postcode{OX11 ORA}, \country{UK}}}
}

\authormark{Aznaran, Farrell, Monroe, \& Van-Brunt}

\corresp[*]{Corresponding author: \href{email:patrick.farrell@maths.ox.ac.uk}{patrick.farrell@maths.ox.ac.uk}}


\received{16\textsuperscript{th}}{09}{2022}

\abstract{We develop finite element methods for coupling the steady-state Onsager--Stefan--Maxwell equations to compressible Stokes flow. These equations describe multicomponent flow at low Reynolds number, where a mixture of different chemical species within a common thermodynamic phase is transported by convection and molecular diffusion.
    Developing a variational formulation for discretizing these equations is challenging: the formulation must balance physical relevance of the variables and boundary data, regularity assumptions, tractability of the analysis, enforcement of thermodynamic constraints, ease of discretization, and extensibility to the transient, anisothermal, and non-ideal settings. To resolve these competing goals, we employ two augmentations: the first enforces the mass-average constraint in the Onsager--Stefan--Maxwell equations, while its dual modifies the Stokes momentum equation to enforce symmetry. Remarkably, with these augmentations we achieve a Picard linearization of symmetric saddle point type, despite the equations not possessing a Lagrangian structure. 
    Exploiting the structure of linear irreversible thermodynamics, we prove the inf-sup condition for this linearization, and identify finite element function spaces that automatically inherit well-posedness. We verify our error estimates with a numerical example, and illustrate the application of the method to non-ideal fluids with a simulation of the microfluidic mixing of hydrocarbons.}

\maketitle

\section{Introduction}\label{sec:intro}

Many fluids consist of mixtures; for example, air is a mixture of nitrogen, oxygen, carbon dioxide, and other species. In many situations, it is not necessary to resolve the motions of the individual species, such as when modelling the flow of air over an aircraft. However, in other contexts, detailed knowledge of the transport of individual species is required. Examples include biological applications, where one may be interested in the transport of oxygen and carbon dioxide in blood, in chemical engineering, where one may be interested in separating or combining the constituents of petroleum, or in electrochemistry, where the performance of a lithium-ion battery is often limited by the transport of lithium ions within an electrolyte. We describe this situation as a \emph{multicomponent flow}, where a fluid is composed of $ 2 \leq n \in \mathbb{N}_+$ distinct chemical species in a common thermodynamic phase.
The primary contributions of our work are a novel variational formulation of a system of equations describing non-ideal miscible isothermal multicomponent flow, a Picard linearization of these equations that possesses symmetric saddle point structure (despite the equations not arising from a Lagrangian), and the numerical analysis of a structure-preserving finite element discretization. In particular, we identify the structural relationships required of the finite element spaces for the different variables that allow for the continuous well-posedness to be inherited automatically in the discretization.

More specifically, this work considers multicomponent flow in the \emph{concentrated} (i.e.~general) solution regime, as opposed to the simpler and more commonly studied \emph{dilute} solution regime.
We now provide an overview of this distinction, for self-containment.
The dilute approximation applies when a single species called the \emph{solvent} (conventionally assigned index $i = n$) is taken to have a concentration very far in excess of the remaining species ($i < n$), each of which is called a \emph{solute}. Thus, if $\Omega\subset\mathbb{R}^{d}$ ($d\in\{2, 3\}$) is the domain containing the mixture, the fluid density $\rho:\Omega\to\mathbb{R}$ varies negligibly with solute content in a dilute solution and approximately coincides with the mass density of the pure solvent. This decouples the flow, the material's bulk motion, from the \emph{mass transport}, the motion of individual molecular constituents comprising the material. One can thus solve for the flow velocity, and then employ this velocity in a system of independent advection-diffusion equations for the mass transport. A typical dilute solution problem at low Mach number (where the density is assumed constant) is to solve:
\begin{subequations}\label{eq:dilute_solution_regime}
    \begin{alignat}{3}
         \frac{\partial (\rho v)}{\partial t} -\Div\left( 2 \eta \varepsilon(v) \right) + \Div\left( \rho v \otimes v \right) + \nabla p &= \rho f, \label{eq:dilute-solution-navier-stokes}\\
        \Div v &= 0, \label{eq:CFDincompressible} \\
        \frac{\partial c_i}{\partial t} + \Div\left(c_i v\right) + \Div J_i &= r_i, \quad && i = 1, \ldots, n - 1, \label{eq:continuity_equations} \\
        J_i &= -D_i \nabla c_i, \quad && i = 1, \ldots, n - 1, \label{eq:ficks_law}
    \end{alignat}
\end{subequations}
where $\Omega$ is bounded Lipschitz, $v: \Omega \to \mathbb{R}^d$ is the flow velocity, $\varepsilon$ the symmetric gradient operator $\varepsilon(v) \coloneqq \frac{1}{2} (\nabla v + (\nabla v)^\top )$, $\eta > 0$ the shear viscosity, $p: \Omega \to \mathbb{R}$ the pressure, $f: \Omega \to \mathbb{R}^d$ the body acceleration induced within $\Omega$ due to the action of external fields, $c_i: \Omega \to \mathbb{R}$ the concentration of solute $i$ in the solvent, $J_i: \Omega \to \mathbb{R}^d$ its diffusive flux, $r_i: \Omega \to \mathbb{R}$ its volumetric rate of generation or depletion, and $D_i > 0$ its Fickian diffusion coefficient. The velocity $v_i: \Omega \to \mathbb{R}^d$ of each individual species is given by
\begin{equation}\label{eq:species_velocities_bulk_velocity}
    c_i v_i = c_i v + J_i,
\end{equation}
decomposing the transport of each species into a convective and a diffusive contribution. The dilute solution regime is characterized by the approximation $v \approx v_n$, that the bulk motion of the fluid coincides with the motion of the solvent.

While the dilute solution approximation has been applied to great effect~\citep{Bird2002, Levich1962, Cussler2009, Deen2016}, it fails starkly when no particular species is present in great excess, the concentrated solution regime of interest in this work. Several problems arise when attempting to relax the dilute solution approximation and formulate and discretize models for concentrated solutions. In concentrated solutions the very notion of `flow velocity' becomes ambiguous, because the overall bulk motion of the fluid need not coincide with any particular species velocity and these species velocities will in general be distinct. One can still identify a natural composition-dependent definition of $v$ in the concentrated case, however. The density of the fluid is given by
\begin{equation}\label{eq:density}
    \rho \coloneqq \sum_{i=1}^{n} M_{i} c_{i},
\end{equation}
in which $M_{i} > 0$ represents the molar mass of species $i$.
Using~\cref{eq:species_velocities_bulk_velocity}, the continuity equations~\cref{eq:continuity_equations} may be rephrased in terms of species velocities as
\begin{equation}\label{eq:cty-eqns}
    \frac{\partial c_i}{\partial t} + \Div(c_i v_i) = r_i, \quad i = 1, \ldots, n.
\end{equation}
Time differentiation of~\cref{eq:density}, followed by elimination of the concentration derivatives with~\cref{eq:cty-eqns}, yields
\begin{equation}\label{eq:partiallyprovedmasscont}
    \frac{\partial \rho}{\partial t} = \sum_{i} M_{i} \left( r_i - \Div \left( c_{i} v_{i} \right) \right) = -\Div \left( \sum_{i} M_{i} c_{i} v_{i} \right) .
\end{equation}
Here, the last equality incorporates the premise that homogeneous chemical reactions conserve atoms, which requires that $\sum_i M_i r_i = 0$.\footnote{In multiphase flows, heterogeneous material exchange may occur, leaving a nonzero generation term in the mass continuity equation for a given phase. We limit the present discussion to single-phase flows.} \Cref{eq:partiallyprovedmasscont} is consistent with the common understanding of mass continuity if the flow velocity $v$ within a multicomponent fluid is identified as the so-called \textit{mass-average velocity}, defined as~\citep[p.~454]{Hirschfelder1954}
\begin{equation}\label{eq:m-a-vel-constraint}
    v \coloneqq \sum_{i} \omega_{i} v_{i},
\end{equation}
where
\begin{equation}\label{eq:mass-frac}
    \omega_{i} \coloneqq \frac{M_{i} c_{i}}{\rho}
\end{equation}
defines the \textit{mass fraction} of species $i$. Indeed, rewriting~\cref{eq:partiallyprovedmasscont} in terms of the mass-average velocity yields
\begin{equation}\label{eq:mass-cty}
    \frac{\partial \rho}{\partial t} = - \Div (\rho v),
\end{equation}
thereby recovering the mass continuity equation familiar from fluid mechanics. In our formulation we will solve for both the mass-average velocity and the individual species velocities. The mass-average velocity is governed by a momentum balance, typically expressed in the form of the Cauchy equation
\begin{equation}\label{eq:Cauchymomentumequation}
    \frac{\partial (\rho v)}{\partial t} - \Div\tau + \Div\left(\rho v\otimes v\right) + \nabla p = \rho f,
\end{equation}
where $\tau: \Omega \to \mathbb{R}^{d \times d}_{\textrm{sym}}$ is the dissipative (viscous) internal stress tensor,\footnote{In general, an internal stress $\tau$ is characterized by $\int_{\partial M} \tau~\ds$ expressing the net force exerted on the surroundings by a volume $M \subset \Omega$ on the closed surface $\partial M$ that bounds it.}
to be specified with a constitutive law. If Newton's law of viscosity is used, then~\cref{eq:Cauchymomentumequation} reduces to the Navier--Stokes momentum equation~\cref{eq:dilute-solution-navier-stokes}.

Another issue to address when moving to the concentrated solution regime is the choice of constitutive law for the diffusive fluxes. In dilute solutions each solute interacts at a molecular level almost solely with solvent molecules, and so the diffusive solute fluxes $J_i$ can each be modelled by Fick's law~\citep{Fick1855}. In concentrated solutions the constitutive laws for mass transport become incomplete, because Fick's law~\cref{eq:ficks_law} fails to take into account all possible species-species interactions. Even in the case of simple diffusion (where $v = 0$ uniformly), the diffusive flux of a given species can generally be driven by a concentration gradient of any other species in the solution---a phenomenon known as \textit{cross-diffusion}. The theory of linear irreversible thermodynamics, pioneered by~\citet{Onsager1931, Onsager1931a, Onsager1945}, enables the thermodynamically consistent generalization of Fick's law~\cref{eq:ficks_law} to the concentrated solution regime. This formalism is described in the next subsection.

\subsection{Onsager--Stefan--Maxwell equations}

Within a multi-species solution, the Onsager--Stefan--Maxwell (OSM) equations relate the \emph{diffusion driving forces} $d_{i}: \Omega \to \mathbb{R}^d$ to the species velocities $v_i$ via
\begin{equation}\label{eq:OSM-eqn}
    d_{i} = \sum_{j} \textbf{M}_{ij} v_{j}, \quad i = 1,\ldots, n,
\end{equation}
where the diffusion driving forces $d_i$ incorporate the effects of various state variable gradients and external forces~\cite[Eq.~(2.5.4)]{Giovangigli1999}, and where
$\textbf{M}$ is the Onsager transport coefficient matrix with entries
\begin{equation}\label{eq:transport-matrix}
    \textbf{M}_{ij}  =
    \left\{
           \renewcommand{\arraystretch}{1.2}
            \setlength{\arraycolsep}{0pt}
            \begin{array}{ l@{~} l}
                -\frac{RT c_{i} c_{j}}{\mathscr{D}_{ij}c_{\text{T}} }&\text{ if }i \neq j, \\
                \sum_{k=1, k \neq i}^{n}\frac{RTc_{i} c_{k}}{\mathscr{D}_{ik}c_{\text{T}} } &\text{ if }i = j.
            \end{array}
     \right.
\end{equation}
Here $R > 0$ is the ideal gas constant, $T > 0$ the ambient temperature, $c_\text{T}$ denotes the total concentration defined as
\begin{equation}\label{eq:c_T}
    c_{\text{T}}: = \sum_{i} c_{i},
\end{equation}
and $\mathscr{D}_{ij} \in \mathbb{R}$ represents the Stefan--Maxwell diffusivity of species $i$ through species $j\neq i$. The Stefan--Maxwell diffusivities are symmetric in the species indices, $\mathscr{D}_{ij} = \mathscr{D}_{ji}$, and $\mathscr{D}_{ii}$ is not defined. In the present discussion we restrict attention to the case where every $\mathscr{D}_{ij}$ is constant, which in turn requires each to be positive~\citep{Van-Brunt2021}, but in general the Stefan--Maxwell diffusivities may depend on the species concentrations, temperature, and pressure. \Cref{eq:OSM-eqn} is often presented as
\begin{equation}\label{eq:OSM-eqn-spelled-out}
    d_i = \sum_{\substack{j=1 \\ j\neq i}}^{n} \frac{RT c_{i} c_{j}}{ \mathscr{D}_{ij} c_{\text{T}}} ({v}_{i} - {v}_{j}),
\end{equation}
which follows from~\cref{eq:OSM-eqn} and~\cref{eq:transport-matrix}.

In an isothermal but nonisobaric medium without external forces, the diffusion driving forces may be identified as~\citep[Eq.~(24.1-8)]{Bird2002}
\begin{equation}\label{eq:driving-force}
    d_{i} \coloneqq -c_{i} \nabla \mu_{i} + \omega_{i} \nabla p, \quad i = 1, \ldots, n,
\end{equation}
in which $\mu_{i}: \Omega \to \mathbb{R}$ is the chemical potential of species $i$.
The chemical potential represents the partial derivative of the Gibbs free energy (a quantity describing the total amount of work a system can deliver to isothermal, isobaric surroundings) with respect to the number of moles of a given species $i$ at constant temperature and pressure, and holding the molar contents of all other species fixed. The chemical potentials are related to the concentrations and pressure via a constitutive law, discussed below in~\cref{sec:chemical-potential-constitutive-law}.

As a consequence of the statistical reciprocal relations developed by~\citet{Onsager1931, Onsager1931a} and the second law of thermodynamics, the transport matrix $\textbf{M}$ is everywhere symmetric positive semidefinite, with a single null eigenvalue associated with $(1, \ldots, 1)^\top$:
\begin{equation}\label{eq:nullspace}
    \sum_{j} \textbf{M}_{ij} = 0.
\end{equation}
This nullspace is necessary to distinguish convection, which is non-dissipative, from diffusion.
A consequence of this nullspace is that one may shift each $v_i$ in~\cref{eq:OSM-eqn} by the mass-average velocity
\begin{equation}\label{eq:OSM_shifted}
    d_i = \sum_{j} \textbf{M}_{ij} \left(v_{j} - v\right), \quad i = 1, \ldots, n,
\end{equation}
so that the transport matrix acts on terms proportional to the diffusive flux $J_i$. Thus~\cref{eq:OSM-eqn} can be understood as an implicit constitutive relation for the diffusive fluxes~\citep{Bulicek2021}.

The symmetry of $\textbf{M}$ combined with~\cref{eq:nullspace} allows one to show that
\begin{equation}\label{eq:gibbs--duhem}
    \sum_{i} d_{i} = \sum_{i} \left(-c_{i} \nabla \mu_{i} + \omega_i\nabla p\right) = 0,
\end{equation}
an expression of the isothermal, nonisobaric Gibbs--Duhem equation from equilibrium thermodynamics.
For further detail on the historical development and mathematical structure of the Onsager--Stefan--Maxwell equations, we refer the reader to~\cite{Van-Brunt2021}. 

\subsection{Stokes equation}\label{sec:stokes-eqn}

A further consequence of the nullspace of $\textbf{M}$~\cref{eq:nullspace} is that the species velocities cannot be determined from the diffusion driving forces alone. They can be recovered, however, by incorporating the Cauchy momentum equation~\cref{eq:Cauchymomentumequation} to specify the mass-average velocity~\cref{eq:m-a-vel-constraint}, which is required for a complete description of the overall transport problem;
this in turn requires a constitutive law for the viscous stress.
For isothermal Newtonian fluids,
the viscous stress $\tau$ relates to the linearized strain rate, the symmetric part of the velocity gradient, through
\begin{equation}\label{eq:viscous-stress-explicit-constit}
    \tau = 2\eta\left(\varepsilon(v) - \frac{\tr\varepsilon(v)}{d}\mathbb{I}\right) + \zeta\tr(\varepsilon(v))\mathbb{I},
\end{equation}
where $\zeta > 0$ is the bulk viscosity, or equivalently
\begin{equation}\label{eq:viscous-stress-constit}
    \varepsilon(v) = \frac{1}{2\eta}\tau + \left(\frac{1}{d^2\zeta} - \frac{1}{2\eta d}\right)(\tr\tau)\mathbb{I} \eqqcolon \mathcal{A}\tau,
\end{equation} 
where $\mathcal{A}:\mathbb{R}^{d\times d}_{\text{sym}}\to\mathbb{R}^{d\times d}_{\text{sym}}$ denotes the compliance tensor.
The full Cauchy stress $\sigma: \Omega \to \mathbb{R}^{d \times d}_{\text{sym}}$ may then be decomposed as
\begin{equation}\label{eq:low-Re-cauchy-stress}
    \sigma = \tau - p\mathbb{I}.
\end{equation}
We further consider steady-state creeping flow, under which assumptions Stokes' equation, 
\begin{equation}\label{eq:stokes}
    \Div\tau - \nabla p = -\rho f,
\end{equation}
follows from the momentum balance~\cref{eq:Cauchymomentumequation}. 

The OSM equations~\cref{eq:OSM-eqn} are written in \textit{force-explicit} form: the
equations express the species
velocities (fluxes) as implicit variables. We choose also to write the Newtonian constitutive equation~\cref{eq:viscous-stress-constit} in this manner, expressing the thermodynamic force (the linearized strain rate) in terms of the corresponding flux (the viscous stress)~\citep{Hirschfelder1954}. Typically
in computational fluid dynamics, a flux-explicit formulation is obtained by using an explicit constitutive relation such as~\cref{eq:viscous-stress-explicit-constit} to eliminate the Cauchy and
viscous stresses in the first instance. For our overall coupled system (stated later in~\cref{eq:SOSM--equations}), we do not eliminate the viscous stress, but include it as an implicit variable to be solved for. While this increases the computational cost, it has
substantial benefits; the viscous stress plays a fundamental role in the calculation of local entropy production, but more significantly, we show in~\cref{sec:nonlinear-problem-formulation} that the resulting system of equations can be cast as a
symmetric perturbed saddle point-like system, which is conducive to both
theoretical analysis and (we anticipate) efficient linear solvers.

\subsection{Augmentation of the diffusion transport matrix and the Stokes momentum balance}\label{sec:augmentation}

The variational formulation of our equations must enforce the relation~\cref{eq:m-a-vel-constraint}, between the bulk (mass-average) velocity and the species velocities.
We employ the augmentation approach introduced by~\citet{Helfand1960} and later used by~\citet{Giovangigli1990, Ern1994, Giovangigli1999}. We augment each OSM equation~\cref{eq:OSM-eqn-spelled-out} by adding the mass-average velocity constraint~\cref{eq:m-a-vel-constraint} to both sides with a prefactor $\gamma > 0$:
\begin{equation}\label{eq:augmented-OSM}
    d_i + \gamma\omega_{i} v = \sum_{j\neq i} \frac{RT c_{i} c_{j}}{ \mathscr{D}_{ij} c_{\text{T}}} ({v}_{i} - {v}_{j}) + \gamma\omega_{i} \sum_{j} \omega_{j} {v}_{j}.
\end{equation} 
The units of $\gamma$ are uniquely determined by~\cref{eq:augmented-OSM}. For convenience, we define the \textit{augmented transport matrix} 
\begin{equation}\label{eq:augmented-transport-matrix}
    \textbf{M}^{\gamma}_{ij} \coloneqq \textbf{M}_{ij} + \gamma\omega_{i} \omega_{j},
\end{equation}
so that~\cref{eq:augmented-OSM} can be stated as $d_i + \gamma\omega_i v = \sum_{j\neq i}\mathbf{M}_{ij}^\gamma v_j$.
We may then compute directly
\begin{equation}
    \sum_{i,j} {v}_{i} \cdot \textbf{M}^{\gamma}_{ij} {v}_{j} = \frac{1}{2} \sum_{j \neq i} \frac{ RT c_{i}c_{j}}{ \mathscr{D}_{ij}c_{\text{T}}} \vert{v}_{i} - {v}_{j}\vert^{2} + \gamma\left\vert \sum_{j} \omega_{j} {v}_{j} \right\vert^{2}
\end{equation}
to show 
that the augmented transport matrix is symmetric positive definite~\citep{Van-Brunt2021}. This was used in~\citet{Van-Brunt2021} to construct coercive bilinear forms for the pure Stefan--Maxwell diffusion problem.
With this augmentation, although the transport matrix is \textit{a priori} singular, one can nevertheless recover the species velocities from the driving forces by coupling with the mass-average velocity constraint~\cref{eq:m-a-vel-constraint}.

The augmentation~\cref{eq:augmented-OSM} modifies a constitutive law of
the system, which will induce coercivity of a certain bilinear form below. However, this
comes at the cost of symmetry. To recover symmetry, we add a
`dual' augmentation to the Stokes
equation~\cref{eq:stokes}
\begin{equation}\label{eq:augmented-stokes}
    \Div\tau - \nabla p = -\rho f + \gamma \sum_{j} \omega_{j}(v - v_{j}). 
\end{equation}
With these two augmentations~\cref{eq:augmented-OSM} and~\cref{eq:augmented-stokes}, an important bilinear form defined later
in~\cref{eq:important-form-A} will be both symmetric and coercive
on an appropriate kernel.
This greatly aids the proofs of well-posedness for the
continuous and discrete problems,
as we will demonstrate in~\cref{sec:cts-well-posed} and~\cref{sec:structural-assump}.

To close the equations, we must relate the chemical potentials to the concentrations via a thermodynamic equation of state, discussed next.

\subsection{The chemical potential and equation of state}\label{sec:chemical-potential-constitutive-law}

Our variational formulation will solve for the chemical potential $\mu_i$ of each species $i$. 
This has several advantages. First, this allows for a general formula for the diffusion driving forces~\cref{eq:driving-force}, independent of the materials considered. If we were to make the (perhaps more obvious) choice of solving for concentrations $c_i$ as the primary variables instead, the form of the diffusion driving forces would change in a material-dependent manner. Second, our choice allows for a decoupling in the linearization we employ: the primary mixed system to solve only depends on the material via the diffusion coefficients and viscosities, with any non-ideality confined to the computation of concentrations and density postprocessed at every iteration using material-dependent thermodynamic constitutive relations discussed below. Third, together with the choice to solve for the viscous stress as described in~\cref{sec:stokes-eqn}, this decoupling endows the equations to solve with a symmetric perturbed saddle point-like structure.

Generally each species concentration $c_i$ can be inferred from $\{\mu_i\}_{i=1}^{n}$ and $p$, given thermodynamically consistent constitutive laws for the chemical potential, and an equation of state which relates $c_{\text{T}}$ to pressure and composition.
Within an isothermal ideal gas, this relation is simply
\begin{equation}\label{eq:ideal-gas-chemical-potential}
    c_i = \frac{p^\ominus}{RT}\exp\left(\frac{\mu_i - \mu^{\ominus}_{i}}{RT}\right),
\end{equation}
for some known reference pressure $p^{\ominus}$ and a set of reference chemical potentials $\{\mu^{\ominus}_{i}\}_i$.
A general relation for non-ideal systems is
\begin{equation}\label{eq:chemical-constitutive-law}
    \mu_{i} = \mu_{i}^{\ominus} + RT \ln (\gamma_{i} x_{i}),
\end{equation}
where $x_i \coloneqq c_i/c_{\text{T}}$ is the \textit{mole fraction}, and $\gamma_{i}$ the \textit{activity coefficient}, of species $i$. (Within a system made up of $n$ species, specifying $n - 1$ mole fractions determines the \emph{composition} referred to earlier.) In non-ideal solutions, the reference potentials $\mu^{\ominus}_{i}$ generally depend on the temperature and pressure~\citep{Guggenheim1985, Atkins2010}; they determine the value of the molar Gibbs free energy of pure species $i$ at the $T$ and $p$ values of interest. Activity coefficients generally depend on temperature, pressure, and composition; the definition of the reference state further requires that they approach unity at infinite dilution, i.e.~$\lim_{x_{i} \rightarrow 0} \gamma_{i} = 1$.
Constitutive laws~\cref{eq:chemical-constitutive-law} suffice to determine the mole fractions (but not concentrations) within non-ideal solutions. To obtain the concentrations, an additional equation of state for the system as a whole is required, which may be expressed in volumetric form as
\begin{equation}\label{eq:volumetric-eqn-of-state}
    c_{\text{T}} = \frac{1}{\sum_{i} V_{i} x_{i}},
\end{equation}
in which $V_{i} > 0$ is the partial molar volume of species $i$. Formally, the partial molar volume is a material property that quantifies the change in total fluid volume with respect to the number of moles of a species $i$ at constant temperature and pressure, holding all other species contents fixed. Maxwell relations derived from the Gibbs free energy also require that $V_{i}$ quantifies the partial derivative of $\mu_{i}$ with respect to $p$~\citep[Eq.~(28)]{Goyal2017}. This dependence, part of which is embedded in the pressure dependence of $\mu_{i}^{\ominus}$, may be regarded as given data that is experimentally measurable~\citep{Doyle1997}. 

Our linearization below is designed so that the concentrations are calculated from the chemical potentials and pressure.
This trivially guarantees positivity of the concentrations, but more significantly, the model is able to incorporate non-ideality by employing chemical potential constitutive laws more general than~\cref{eq:ideal-gas-chemical-potential},
such as~\cref{eq:chemical-constitutive-law}.

\subsection{Coupled problem statement}

Our goal is to find and analyze a variational formulation and structure-preserving finite element discretization of the following problem: 
given data $f$ and $\{r_i\}_{i=1}^n$,
find
chemical potentials $\{\mu_{i}\}_{i=1}^n$, viscous stress $\tau$, pressure $p$, species velocities $\{v_i\}_{i=1}^n$, and convective velocity $v$ satisfying
\begin{subequations}\label{eq:SOSM--equations}
    \begin{align}
        -c_{i} \nabla \mu_{i} + \omega_{i} \nabla p + \gamma\omega_{i} v &= \sum_{j} \textbf{M}^{\gamma}_{ij} v_{j}~\forall i, && \text{(augmented OSM equations)} \label{eq:SOSM--equations-augmented-OSM}\\ 
        \varepsilon(v) &= \mathcal{A} \tau, && \text{(stress constitutive law)} \label{eq:SOSM--equations-stress-constit}\\ 
        \Div \tau - \nabla p -\gamma \sum_{j} \omega_{j}(v - v_{j}) &= -\rho f, && \text{(augmented Stokes equation)} \label{eq:SOSM--equations-Stokes}\\ 
        \Div (c_{i} v_{i}) &= r_{i}~\forall i, && \text{(species continuity equation)} \label{eq:SOSM--equations-continuity}\\ 
        \Div(v) &= \Div \left( \sum_{j} \omega_{j} v_{j} \right), && \text{(mass-average velocity constraint)} \label{eq:SOSM-divergence-of-mass-average-velocity}
    \end{align}
\end{subequations}
for an augmentation parameter $\gamma \geq 0$, where $\{c_i, \omega_i\}_{i=1}^{n}, \rho$ are functions of the unknowns via chemical potential constitutive laws such as~\cref{eq:chemical-constitutive-law} and~\cref{eq:volumetric-eqn-of-state}, and algebraic relations~\cref{eq:density},~\cref{eq:mass-frac}. We shall introduce appropriate boundary conditions in~\cref{sec:nonlinear-variational}.
We call the system~\cref{eq:SOSM--equations} the (augmented) \textit{Stokes--Onsager--Stefan--Maxwell} (SOSM) system. When the convection term $\Div \left( \rho v \otimes v \right)$ is incorporated into~\cref{eq:SOSM--equations-Stokes}, we call this the~\textit{Navier--Stokes--Onsager--Stefan--Maxwell} (NSOSM) system.
Note in the system we only directly enforce the \textit{divergence} of the mass-average velocity constraint~\cref{eq:SOSM-divergence-of-mass-average-velocity}, which may be interpreted
as the compressible generalization of the standard divergence
constraint~\cref{eq:CFDincompressible}; 
this choice gives rise to a saddle point-like structure, as we show in the next section. 
Nevertheless, the full constraint~\cref{eq:m-a-vel-constraint} is incorporated via the augmentations~\cref{eq:augmented-OSM} and~\cref{eq:augmented-stokes}, as discussed further in~\Cref{rem:enforcement--of--MAVC}.

\subsection{Relation to existing literature and outline}

For dilute solutions with constant solvent concentration (and no volumetric generation or the depletion of the solvent, $r_n = 0$), the NSOSM equations reduce to the incompressible Navier--Stokes equations, as well as convection-diffusion equations constituted by Fick's law for each solute. These equations have been studied for many decades, and effective numerical techniques are available.
We do not attempt a systematic review here, but mention~\citet{Stynes2018, Hundsdorfer2013, Thomee2006, Elman2014} as gateways to this literature.
In this regime, the momentum solve and the equation for the transport of concentration are decoupled using incompressibility.

Our formulation~\cref{eq:SOSM--equations} solves for the viscous stress as an unknown variable.
Of most relevance to this aspect of our approach is the work of~\citet{Carstensen2012}, who discretized the stress in an incompressible
stress-velocity Stokes system using the same stress elements of~\citet{Arnold2002} that we shall employ.

Related systems of equations have been formulated and analyzed, including the coupling of the Stefan--Maxwell equations with the incompressible Navier--Stokes equations by~\citet{Chen2015}, the compressible Navier--Stokes equations by~\citet{Bothe2021}, the Darcy momentum equation by~\citet{Ostrowski2020}, and the Cahn--Hilliard equations by~\citet{Huo2022}. The coupling of an anisothermal NSOSM system to surface phenomena by sorption was formulated by~\citet{Soucek2019}; \citet{Bothe2015} proposed an extension to the Stefan--Maxwell equations in the presence of chemically reacting constituents.

Numerical methods for solving the NSOSM equations have received much less attention. The only works of which the authors are aware are those of Ern, Giovangigli, and coauthors, including a monograph~\citep{Ern1994} and a series of other works~\citep{Giovangigli1999, Ern1998, Ern1999, Ern2004} which apply multicomponent transport to combustion modelling for ideal gas mixtures. These schemes use sophisticated finite difference methods, with the important exception of~\citet{Ern2004}, which uses a finite element method with additional least-squares terms to stabilize the formulation.
The authors are unaware of any literature that addresses numerical methods for SOSM or NSOSM systems in the non-ideal case.

For OSM models of isobaric ideal gases under simple diffusion, several recent papers have addressed numerical approaches, including a finite element method proposed by~\citet{McLeod2014}, a finite volume method by~\citet{Cances2020}, and a finite difference scheme by~\citet{Bondesan2019}.

Recently a finite element scheme for simple isobaric OSM diffusion in ideal gases was proposed by a subset of the current authors~\citep{Van-Brunt2021}. The approach solved the augmented OSM equations~\cref{eq:augmented-OSM} combined with the species continuity equations~\cref{eq:SOSM--equations-continuity}. This paper builds on the foundation established in~\citet{Van-Brunt2021}, but now fully incorporates momentum, non-ideality, and pressure-driven diffusion.
In contrast to this prior work, we are able here to avoid a generalized saddle point formulation, and in~\cref{sec:linearization} will derive a symmetric perturbed saddle point system to be solved at each nonlinear iteration---a more classical linear algebraic structure for which many solvers have been developed~\citep{Benzi2005}. 
However, due to the more complex form~\cref{eq:driving-force} of the driving force, and since we solve for the chemical potentials to allow for non-ideal fluids, we are not able to enforce the Gibbs--Duhem relation~\cref{eq:gibbs--duhem} to machine precision, as achieved in~\cite{Van-Brunt2021}.

\begin{remark}\label{rem:no-lagrangian}
    Many cross-diffusion systems, such as those describing multiagent systems in mathematical biology~\citep{Carrillo2018}, arise from a gradient flow of an associated entropy functional. Unfortunately, although the OSM system admits an associated thermodynamic energy---the Gibbs free energy---we are not able to show equivalence of the (S)OSM system to the Euler--Lagrange stationarity condition of any energy or Lagrangian functional, and hence cannot exploit any gradient flow structure. Instead, our mathematical line of attack will be to exploit the positive definiteness of the augmented transport matrix $\mathbf{M}^\gamma$~\cref{eq:augmented-transport-matrix}. With our augmentations of the equations, the Picard scheme we propose below in~\Cref{sec:linearization} nevertheless gives rise to symmetric linearized problems to solve at each nonlinear iteration.
\end{remark}

The remainder of this work is organized as follows. In~\Cref{sec:nonlinear-problem-formulation}, we derive a novel variational formulation of the fully coupled nonlinear SOSM problem, incorporating boundary conditions and augmentation terms, as a nonlinear perturbed saddle point-like system, using a novel solution-dependent test space relating to the thermodynamic driving force;
our principal discovery is the duality between the diffusion driving forces, and the combination of species continuity equations with the divergence of the mass-average velocity constraint.
\Cref{sec:linearization} proposes a Picard-like linearization, which is proven to be well-posed under physically reasonable assumptions.
\Cref{sec:discretization} identifies appropriate finite element spaces, and the structural relations which should hold between them, for a well-posed and convergent discretization of this linearization. We then validate our convergence results numerically.
Finally, we illustrate our method by simulating the steady mixing of liquid benzene and cyclohexane in a two-dimensional microfluidic laminar-flow device.

\section{Variational formulation}\label{sec:nonlinear-problem-formulation}

We employ standard notation for the Sobolev space $H^k(\Omega;
\mathbb{X})$ (or $L^2(\Omega; \mathbb{X})$ when $k = 0$) with domain
$\Omega\subset\mathbb{R}^d$ and codomain $\mathbb{X}$, and associated
norm $\|\cdot\|_k$ and seminorm $|\cdot|_k$. We denote by
$\mathbb{S} = \mathbb{R}^{d\times d}_{\text{sym}}$ the space of $d\times
d$ symmetric tensors.  The symbol $\lesssim$ denotes inequality up to a
constant which may depend on mesh regularity but not mesh spacing $h$.
Let $L^2_0(\Omega) \coloneqq \{z\in L^2(\Omega)~|~\dashint_\Omega z~\dx
= 0\}$.  We use the notation $\tilde{q} = (q_{1}, \ldots, q_{n})$ to
denote an $n$-tuple of functions.  Let $\Gamma = \partial\Omega$ and let
$\langle\cdot,\cdot\rangle_{\Gamma}$ denote the $(H^{-1/2}\times
H^{1/2})(\Gamma;\mathbb{R}\text{ or }\mathbb{R}^d)$ dual pairing.

\subsection{Integrability of pressure gradients}

For isothermal, isobaric multicomponent diffusion as originally considered by~\citet{Maxwell1867} and~\citet{Stefan1871},
it suffices to work with driving forces of the form
\begin{equation}\label{eq:driving-force-ideal}
    d_i = -RT \nabla c_i. 
\end{equation}
In a variational formulation of the nonisobaric case, one would like to integrate the pressure gradient term in our diffusion driving forces~\cref{eq:driving-force} by parts, to reduce the regularity requirement on $p$. However,
it is not obvious how to do so, since the mass fractions $\omega_i$ are spatially varying.
In order to rigorously incorporate the effect of pressure-driven diffusion, we are therefore led to consider the somewhat unorthodox possibility of formulating the Stokes subproblem with pressure lying in $H^1(\Omega)$.
Typically, the condition that $p \in H^1(\Omega)$ may be provided by elliptic regularity results for
the pressure field,
but to the authors' knowledge, the \textit{a priori} square-integrability of pressure gradients 
(i.e.~for which, we emphasize, pressure is \textit{defined} to lie in $H^1(\Omega)$)
has not been considered for the Stokes system, except at the discrete level for the incompressible case in~\cite{Stenberg1989}.
This condition is also suggested by the case of pure
Stefan--Maxwell diffusion for nonisobaric ideal gases. Here the
driving forces are
\begin{equation}\label{eq:driving-forces-ideal-nonisobaric}
    d_i = -RT \nabla c_i + \omega_i \nabla p,
\end{equation}
which suggests considering each $c_i$ (and hence $c_{\text{T}}$) to lie
in $H^1(\Omega)$, which forces the pressure to lie in the same space due to the ideal equation of state $p = c_{\text{T}}RT$.

In general, one must distinguish between the \textit{thermodynamic} pressure $p$, which we use throughout this paper,
and the \textit{mechanical} pressure $p_\text{m} \coloneqq -{\tr\sigma}/{d}$.
The mechanical pressure is related to the spherical Cauchy stress by $\sph\sigma \coloneqq \frac{\tr\sigma}{d}\mathbb{I} = -p_\text{m}\mathbb{I}$, and to $p$ by
\begin{equation}\label{eq:2-pressures}
    p = p_\text{m} + \zeta\Div v.
\end{equation}
In the context of multicomponent flow, this decomposition is discussed in further detail by~\citet{Bothe2015}.
Even in the simpler incompressible limit where $p = p_\text{m}$, we cannot expect extra regularity of $\nabla p = -\Div(\sph\sigma)$ because $\hdivs$, the natural space for $\sigma$, is not closed under taking spherical 
parts.\footnote{In any case, the incompressible regime for which $\rho$ is constant is physically irrelevant to the OSM framework for mass diffusion, which exhibits spatial heterogeneity of the density. We also remark that, viewing the pressure as a component of the full Cauchy stress, appealing to the Hodge decomposition of the stress space $H(\Div;\mathbb{S})$~\citep[Theorem 4.5]{Arnold2018} does not endow that component with any extra regularity.}
Consequently, we do not take $p\in H^1(\Omega)$, but as a compromise consider a weaker condition defined by the combined viscous stress-pressure space
\begin{equation}\label{eq:stress-pressure-space}
    \{(\tau, p)\in L^2(\Omega;\mathbb{S})\times L^2(\Omega)~\vert~\Div \tau - \nabla p\in L^2(\Omega;\mathbb{R}^d)\} \left(\supsetneq \hdivs\times H^1(\Omega)\right),
\end{equation}
and assign to it the weaker norm $\|\tau\|_{0}^2 + \|p\|^2_{0} +
\|\Div\tau - \nabla p\|_{0}^2$. This space and norm were previously
employed by~\cite{Manouzi2001} in an analysis of a non-Newtonian incompressible Stokes flow where $\tau$ was taken to be the deviatoric shear
stress. Membership of the space~\cref{eq:stress-pressure-space}
is naturally interpretable as the square-integrability of the divergence of the full Cauchy
stress, i.e.~that $\sigma = \tau - p\mathbb{I}\in\hdivs$. Together with
an analogous condition for the chemical potential gradient to be
detailed next, this weaker condition will
account for the pressure gradient in the driving forces.\footnote{One
alternative approach is provided by attempting to construct a smoother
analogue of the stress elasticity complex associated with the Cauchy
stress space~\cref{eq:stress-pressure-space}, for
which~\cref{eq:stress-pressure-space} is replaced by some superspace of
$\hdivs\times H^1(\Omega)$, just as the Stokes complex is precisely a
smoothing of the de Rham complex. We do not pursue this.}

\subsection{Fully coupled variational formulation}\label{sec:nonlinear-variational}

In this subsection, we derive a variational formulation for the stationary problem as a nonlinear perturbed saddle point-like system.
We have found the following statement of the problem to be a feasible tradeoff between the (competing) goals of: physical relevance of variables and boundary data, regularity assumptions, numerical implementability and effectiveness, analytic tractability, enforcement of fundamental thermodynamic relations, and extensibility to the transient, anisothermal, and non-ideal settings.

For boundary data,
we prescribe
mass flux and 
molar fluxes:
\begin{subequations}
    \begin{alignat}{2}
        \rho v &= g_{\text{v}} \in H^{1/2}(\Gamma;\mathbb{R}^d) &&\text{ on }\Gamma, \\ 
        c_i v_{i} \cdot \bn &= g_{i} \in H^{-1/2}(\Gamma) &&\text{ on }\Gamma, \quad i = 1, \ldots, n.
    \end{alignat}
\end{subequations}
For consistency with the mass-average velocity constraint~\cref{eq:m-a-vel-constraint}, we require
\begin{equation}\label{eq:boundary-mass-avg-velocity}
    \sum_{i} M_{i} g_{i} = g_{\text{v}} \cdot \bn,
\end{equation}
with equality in $H^{-1/2}(\Gamma)$.
We further impose conditions
\begin{align}
     \int_{\Omega} p~\dx = \int_{\Omega} \mu_{i}~\dx = 0, \quad i = 1, \ldots, n,
\end{align}
on the pressure and chemical potentials. Typically, the equation of state will require or imply strict positivity of $p$ everywhere, in which case this condition should be understood as $\dashint_\Omega p~\dx = p^{\ominus}> 0$ and that $p$ be shifted by the known value $p^{\ominus}$ as a postprocessing step.

Let $Q = L^2(\Omega;\mathbb{R}^d)^n\times L^2(\Omega;\mathbb{R}^d)$ with norm $\|(\tilde{v}, v)\|_Q^2 \coloneqq \|\tilde{v}\|_0^2 + \|v\|_0^2$.
\sloppy For formal derivation of the weak form, we assume the solution tuple $(\tilde{\mu}, \tau, p, \tilde{v}, v)$ to be smooth on $\overline{\Omega}$, and
consider choosing $(\tilde{w}, s, q)$ from the solution-dependent potential-stress-pressure test space
\begin{equation}\label{eq:test-space}
    \setlength{\arraycolsep}{1pt}
    \Theta \coloneqq \left\{ \left(\tilde{w}, s, q \right) \in L^2_0(\Omega)^{n} \times L^2(\Omega;\mathbb{S})\times L^2_0(\Omega)~\middle|\begin{array}{rl}\Div s - \nabla q &\in L^2(\Omega;\mathbb{R}^{d}),\\
    - c_{i} \nabla w_{i} + \omega_{i}\nabla q &\in L^2(\Omega;\mathbb{R}^d)~\forall i\\
    \end{array}
    \right\}.
\end{equation}
Here it is understood that the $\{c_i, \omega_i\}_{i}$ are computed from the
solution tuple.
Multiplying the $i^{\text{th}}$ continuity equation~\cref{eq:SOSM--equations-continuity} by $w_i$, the divergence of the mass-average velocity constraint~\cref{eq:SOSM-divergence-of-mass-average-velocity} by $q$, and contracting the stress constitutive law~\cref{eq:SOSM--equations-stress-constit} with $s$, we obtain
\begin{equation}\label{eq:topline-before-ibp}
    \sum_{i}\left(\Div(c_iv_i) - r_i\right)w_i + \Div\left(v - \sum_{i}\omega_iv_i\right)q + (\mathcal{A}\tau - \varepsilon(v)):s = 0,
\end{equation}
and hence
\begin{equation}
    \begin{aligned}
        \int_\Omega \sum_i (\Div(c_iv_i)w_i - \Div(\omega_iv_i)q) + \mathcal{A}\tau:s - (s - q\mathbb{I}):\varepsilon(v)~\dx = \int_\Omega\sum_i r_i w_i~\dx.
    \end{aligned}
\end{equation}
Integrating by parts yields
\begin{equation}\label{eq:weak-topline}
    \begin{aligned}
        \int_\Omega &\mathcal{A}\tau:s + \sum_i(-c_{i} \nabla w_{i} + \omega_{i} \nabla q)\cdot v_i + (\Div s - \nabla q)\cdot v~\dx \\
        &= \left\langle (s - q\mathbb{I})\bn, v\right\rangle_{\Gamma} + \sum_{i}\left\langle c_iv_i\cdot\bn, -w_i + \frac{\omega_i}{c_i}q\right\rangle_{\Gamma} + \int_\Omega \tilde{r}\cdot\tilde{w}~\dx\\
        &= \left\langle (s - q\mathbb{I})\bn, \frac{g_{\text{v}}}{\rho}\right\rangle_{\Gamma} + \sum_{i}\left\langle g_i, -w_i + \frac{M_i}{\rho}q\right\rangle_{\Gamma} + \int_\Omega \tilde{r}\cdot\tilde{w}~\dx.
    \end{aligned}
\end{equation}
Now for each $i = 1, \ldots, n$, we take the scalar product of ${u}_{i} \in L^{2}(\Omega;\mathbb{R}^d)$ with the augmented OSM equation~\cref{eq:SOSM--equations-augmented-OSM} and integrate over $\Omega$ to obtain
\begin{equation}\label{eq:OSM-weak}
    \int_{\Omega}\left(-c_{i} \nabla \mu_{i} + \omega_{i} \nabla p\right) \cdot {u}_{i} - {u}_{i} \cdot \sum_{j} \textbf{M}_{ij} {v}_{j} - \gamma\omega_{i} \sum_{j} \omega_{j} (v_{j} - v) \cdot u_{i}~\dx = 0.
\end{equation}
Taking the inner product of the augmented Cauchy momentum balance~\cref{eq:SOSM--equations-Stokes} with 
$u\in L^2(\Omega;\mathbb{R}^d)$
yields
\begin{equation}\label{eq:cauchy-momentum-weak}
    \int_{\Omega} (\Div \tau - \nabla p) \cdot u - \gamma \left(\sum_{j} \omega_{j} (v - v_j) \right)\cdot u~\dx = -\int_{\Omega} \rho f \cdot u~\dx.
\end{equation}
We sum~\cref{eq:OSM-weak} over $i$ and add~\cref{eq:cauchy-momentum-weak} to derive
\begin{equation}\label{eq:weak-bottomline}
    \begin{aligned}
        &\int_{\Omega} \sum_{i} \left(-c_{i} \nabla \mu_{i} + \omega_{i} \nabla p\right) \cdot u_{i} + \left(\Div \tau - \nabla p \right)\cdot u \\
                                                                &- \sum_{i,j} {u}_i\cdot\textbf{M}_{ij} {v}_{j} - \gamma \left(\sum_{i} \omega_{i} (v_{i} - v) \right) \cdot \left( \sum_{j}\omega_{j} (u_{j} - u) \right)~\dx = \int_\Omega -\rho f \cdot u~\dx.
    \end{aligned}
\end{equation}
Note that both augmentations~\cref{eq:augmented-OSM}
and~\cref{eq:augmented-stokes} were involved in deriving this
expression.

Finally, we 
observe that
by definition, we have $\omega_i\in L^\infty(\Omega)$ with $\|\omega_i\|_{L^\infty(\Omega)}\leq 1$.
Moreover, we make the
physically reasonable assumptions
that the
concentrations 
associated with the solution
are uniformly bounded, $c_i \in L^\infty(\Omega)$,
with 
$c_{i} \geq \kappa > 0$ a.e.~(almost everywhere), as in~\cite{Van-Brunt2021}
(which in turn implies $\mathbf{M}^\gamma_{ij}, \rho\in L^\infty(\Omega)$, and $\rho \geq \kappa\sum_i M_i > 0 $ a.e.), 
and that the density gradient is uniformly bounded, $\nabla\rho\in L^\infty(\Omega;\mathbb{R}^d)$.\footnote{
A comparable condition, that $\rho\in (H^1\cap W^{1,\infty})(\Omega)$ is bounded below with $\frac{\nabla\rho}{\rho}\in L^\infty(\Omega;\mathbb{R}^d)$, was used to analyze a compressible Stokes flow in~\cite{Caucao2016}.
}

\begin{definition}\label{defn:weak-SOSM}
    We define a \emph{weak solution to the augmented Stokes--Onsager--Stefan--Maxwell system} to be a $(2n+3)$-tuple 
    \begin{equation}
        (\{\mu_{i}\}_{i=1}^n, \tau, p, \{v_i\}_{i=1}^n, v) \in L^2_0(\Omega)^n \times L^2(\Omega;\mathbb{S}) \times L^2_0(\Omega)\times \underbrace{L^{2}(\Omega;\mathbb{R}^d)^n\times L^2(\Omega;\mathbb{R}^d)}_{Q}
    \end{equation}
   inducing concentrations $\{c_i\}_{i=1}^n$ through a constitutive law (such as~\cref{eq:ideal-gas-chemical-potential}) 
    implicitly defining $c_{i} = c_{i}(\{\mu_i\}_{i=1}^n, p) \geq \kappa > 0$ a.e.~for $i = 1, \ldots, n$,
    such that
    \begin{subequations}
        \begin{alignat}{2}
            \|c_i\|_{L^\infty(\Omega)} &<\infty,~i = 1, \ldots, n,\\
            \|\nabla \rho\|_{L^\infty(\Omega;\mathbb{R}^d)} &<\infty,\\
            \| \Div \tau - \nabla p \|^{2}_{0} &<\infty,\label{eq:stressintcond}\\
            \| - c_{i} \nabla \mu_{i} + \omega_{i} \nabla p \|^{2}_{0} &< \infty,~i = 1, \ldots, n,\label{eq:drivingintcond}
        \end{alignat}
    \end{subequations}
and satisfying~\cref{eq:weak-topline},~\cref{eq:weak-bottomline} for all test tuples $(\{w_i\}_{i=1}^n, s, q, \{u_i\}_{i=1}^n, u)\in \Theta \times Q$, where $\Theta$ is defined in~\cref{eq:test-space}.
\end{definition}
Observe that the solution tuple does not reside in any standard
Sobolev space, but that the regularity assumptions placed on the solution tuple and test spaces 
ensure that the 
surface terms in~\cref{eq:weak-topline} are well-defined.
Recall that condition~\cref{eq:stressintcond} is the
square-integrability of the Cauchy stress $\sigma$ (as in~\cite{Manouzi2001}).
The nonlinear integrability condition~\cref{eq:drivingintcond} is to our knowledge a novel requirement, but also has a natural interpretation, namely the square-integrability of the diffusion driving forces:\footnote{We conjecture 
that one could alternatively derive a formulation of the SOSM system
dual to ours which takes $d_i\in L^2(\Omega;\mathbb{R}^d)$ as a primary unknown.
We also conjecture that the integrability assumptions
in~\Cref{defn:weak-SOSM} could potentially be relaxed.}
\begin{equation}
    d_{i} \in L^{2}(\Omega;\mathbb{R}^d).
\end{equation}
Moreover, we emphasize that this unorthodox formulation allows the rigorous incorporation of pressure diffusion via the pressure gradient on the left side of~\cref{eq:SOSM--equations-augmented-OSM}, despite the fact that the pressure field is not \textit{a priori} $H^1$-regular in the Stokes subsystem.
Later in~\cref{sec:numerics} we observe convergence of the diffusion driving forces in $L^2$ and of the pressure in $H^1$, 
but otherwise leave this consideration, and
further investigation into the optimal nonlinear formulation of the SOSM system, as intriguing open questions.

\begin{remark}\label{rem:enforcement--of--MAVC}
    In the derivation of~\cref{eq:topline-before-ibp}, we used the divergence of the mass-average velocity constraint~\cref{eq:SOSM-divergence-of-mass-average-velocity}, which ignores the
    $\curl$ component in the Helmholtz decomposition of the
    mass-average velocity relationship~\cref{eq:m-a-vel-constraint}. This choice ensures that the number of equations matches the number of unknown variables. The full constraint
    is weakly incorporated, however, via the augmentations~\cref{eq:augmented-OSM}
    and~\cref{eq:augmented-stokes}. Indeed, provided the constitutive equations relating concentrations and chemical potentials are thermodynamically rigorous in the sense that they arise from a Gibbs free energy functional, one can show using the first law of thermodynamics and the extensivity of the Gibbs free energy that the augmentations will enforce the full constraint~\cref{eq:m-a-vel-constraint}. In the case of an ideal
    isobaric isothermal gas, a proof of this was given
    in~\cite{Van-Brunt2021}. As a consequence, although strict positivity of $\gamma$ is required to make the system well-posed, the value chosen should not in principle affect the solution.
\end{remark}

\section{Linearization and well-posedness}\label{sec:linearization}

\subsection{Variational formulation of a generalized Picard scheme}

In this section we derive a 
variational formulation of a generalized Picard linearization. Given a previous estimate for the potentials $\tilde{\mu}^{k}$ and pressure $p^k$ for $k\geq 0$, we regard these as fixed quantities which determine the concentrations $\tilde{c}^k$ via chemical potential constitutive laws and an appropriate equation of state such as~\cref{eq:ideal-gas-chemical-potential}. This in turn determines the density $\rho^k$, mass fractions $\tilde{\omega}^k$, total concentration $c^k_{\text{T}}$, and transport matrix $\mathbf{M}^k$ defined via~\cref{eq:density},~\cref{eq:mass-frac},~\cref{eq:c_T}, and~\cref{eq:transport-matrix}, respectively.
We then construct a linear system to solve for the next iterate $((\tilde{\mu}^{k+1}, \tau^{k+1}, p^{k+1}), (\tilde{v}^{k+1}, v^{k+1}))$.
This update strategy is expected to be effective because the gradients of chemical potential, pressure, and mass-average velocity primarily drive the dynamics of multicomponent flow; the role of the species concentrations is mostly confined to the effect of altering the drag coefficients in the transport matrix.
We make the following physically reasonable assumptions about each iterate, in analogy to~\Cref{defn:weak-SOSM}.
\begin{assumption}\label{assump:well-posedness-assumption}
    For each $k\geq 0$, we assume $c^k_i \in L^{\infty}(\Omega), \rho^k \in W^{1, \infty}(\Omega)$, and that $c^k_i\geq\kappa > 0$ a.e.~for each $i$.
\end{assumption}
\noindent This again implies 
$\rho^k \geq \kappa\sum_iM_i > 0$ a.e.
We also assume henceforth that $\gamma > 0$.

Given $\tilde{c}^k$ and the corresponding $\tilde{\omega}^k$, we define the iteration-dependent weighted function space
\begin{equation}\label{eq:theta-space}
    \hspace{-1mm}
    \Theta^{k} \coloneqq \left\{ (\tilde{w}, s, q) \in L^2_0(\Omega)^{n} \times L^2(\Omega;\mathbb{S})\times L^2_0(\Omega)\middle|\begin{array}{rl}\Div s - \nabla q&\in L^2(\Omega;\mathbb{R}^{d}),\\
        -c_{i}^{k} \nabla w_{i} + \omega_{i}^{k} \nabla q&\in L^2(\Omega;\mathbb{R}^d)~\forall i
    \end{array}
    \right\},
\end{equation}
whose defining conditions linearize those in~\cref{eq:test-space}. 
This mixed space is Hilbertian with graph norm
\begin{equation}\label{eq:theta-norm}
\begin{aligned}
    \hspace{-1mm}\|(\tilde{w}, s,q)\|_{\Theta^{k}}^2 \coloneqq \|s\|_{0}^2 + \|q\|_{0}^2 +  \|\Div s - \nabla q\|_{0}^2 + \sum_{i} \left(\| w_{i} \|^{2}_{0} +  
                                            \|-c_{i}^{k} \nabla w_{i} + \omega_{i}^{k} \nabla q \|_{0}^{2}\right).
\end{aligned}
\end{equation}
We now formulate our linearized problem as a symmetric perturbed saddle point problem.
Define $A_k:Q\to Q^*, \Lambda:\Theta^k\to(\Theta^k)^*, B_k:\Theta^k\to Q^*$ by
\begin{subequations}
    \begin{alignat}{2}
        A_k(\tilde{v}, v; \tilde{u}, u) &\coloneqq \int_\Omega \sum_{i,j} v_i\cdot\mathbf{M}^k_{ij}u_j~\dx + \gamma \int_{\Omega} \left(\sum_i\omega^k_i(v_i - v)\right)\cdot\left(\sum_j\omega^k_j(u_j - u)\right)\dx, \label{eq:important-form-A}\\
        \Lambda(\tilde{\mu}, \tau, p; \tilde{w}, s, q) &\coloneqq \int_\Omega\mathcal{A}\tau:s~\dx, \\
        B_k(\tilde{\mu}, \tau, p; \tilde{u}, u) &\coloneqq \int_\Omega\sum_{i}(-c^k_i\nabla\mu_i + \omega^k_i\nabla p )\cdot u_i + (\Div\tau - \nabla p)\cdot u~\dx,
    \end{alignat}
\end{subequations}
and the functionals
\begin{equation}
    \begin{aligned}
        \ell^1_k(\tilde{w}, s, q) &\coloneqq \left\langle(s - q\mathbb{I})\bn, \frac{g_{\text{v}}}{\rho^k}\right\rangle_{\Gamma} + \sum_{i}\left\langle g_i, - w_i + \frac{M_i}{\rho^k}q\right\rangle_{\Gamma} + \int_\Omega\tilde{r}\cdot\tilde{w}~\dx,\\
        \ell^2_k(\tilde{u}, u) &\coloneqq -\int_\Omega\rho^k f\cdot u~\dx.
    \end{aligned}
\end{equation}
Note that under~\Cref{assump:well-posedness-assumption}, each of the bilinear functionals is continuous; we will denote their norms as $\|A_{k}\|$, $\| \Lambda\|$, and $\|B_{k}\|$, respectively. 
Our linearized problem is posed as follows: find $((\tilde{\mu}^{k+1}, \tau^{k+1}, p^{k+1}), (\tilde{v}^{k+1},v^{k+1})) \in \Theta^{k} \times Q$ such that
\begin{equation}\label{eq:saddlepoint}
    \begin{aligned}
        \Lambda(\tilde{\mu}^{k+1}, \tau^{k+1}, p^{k+1}; \tilde{w}, s, q) + B_{k}(\tilde{w}, s, q; \tilde{v}^{k+1}, v^{k+1}) &= \ell^{1}_{k}(\tilde{w}, s, q)&&~\forall~(\tilde{w}, s, q)\in \Theta^k, \\
        B_k(\tilde{\mu}^{k+1}, \tau^{k+1}, p^{k+1}; \tilde{u}, u) - A_k(\tilde{v}^{k+1}, v^{k+1}; \tilde{u}, u) &= \ell^{2}_{k}(\tilde{u}, u)&&~\forall~(\tilde{u}, u)\in Q,
    \end{aligned}
\end{equation}
\sloppy i.e., defining the transpose $B_k^\top:Q\to(\Theta^k)^*$ in the canonical way, 
\begin{equation}\label{eq:abstract-linearized}
    \nicemat
            \Block{3-3}{\Lambda} & & & \Block{3-2}{B_k^\top} & \\
            & & & &\\
            & & & &\\
            \hline
            \Block{2-3}{B_k} & & & \Block{2-2}{-A_k}&\\
            & & & &
        \end{bNiceArray}
        \nicevec
                \tilde{\mu}^{k+1} \\ \tau^{k+1} \\ p^{k+1}
            \\
            \hline
                \tilde{v}^{k+1} \\ v^{k+1}
        \end{bNiceArray}
        =
        \nicevec
            \Block{3-1}{
                \ell^1_k
            }\\
            \\
            \\
            \hline
            \Block{2-1}{
                \ell^2_k
            }\\
            \\ 
    \end{bNiceArray}.
\end{equation}
We note that the variational terms involving chemical potential and pressure gradients are precisely of the same variational form as the species continuity equations and the divergence of the mass-average velocity constraint, which can be seen by inspecting~\cref{eq:weak-topline} and~\cref{eq:weak-bottomline}. This key insight is what leads to a symmetric system.

Our nonlinear iteration scheme is as follows: for an initial estimate of the concentrations $\tilde{c}^{0}$, we solve the system~\cref{eq:saddlepoint} for the updated variables $((\tilde{\mu}^{k+1}, \tau^{k+1}, p^{k+1}), (\tilde{v}^{k+1},v^{k+1})) \in \Theta^{k} \times Q$, for $k = 0, 1, 2, \ldots$. By the relations detailed in~\cref{sec:chemical-potential-constitutive-law}, these variables are used to calculate the updated concentrations $\tilde{c}^{k+1}$. This is iterated until for some set tolerance $\varepsilon > 0$,
\begin{equation}
	\left( \|(\tilde{\mu}^{k+1}, \tau^{k+1}, p^{k+1} ) - (\tilde{\mu}^{k}, \tau^{k}, p^{k} ) \|^{2}_{\Theta^{k}} + \| (\tilde{v}^{k+1}, v^{k+1}) - (\tilde{v}^{k}, v^{k}) \|^{2}_{Q} \right)^{1/2} \leq \varepsilon.
\end{equation}

\subsection{Well-posedness of the linearized system}\label{sec:cts-well-posed}

We will now prove that the saddle point system~\cref{eq:saddlepoint} is well-posed under~\Cref{assump:well-posedness-assumption}.
This will require a Poincar\'e-type inequality for the following seminorm on $\Theta^{k}$:
\begin{equation}
    |(\tilde{w}, s, q) |_{\Theta^{k}}^{2}  : =\|s\|_{0}^2 +\| \Div s - \nabla q\|_{0}^2 + \sum_{i} \|-c_{i}^{k} \nabla w_{i} + \omega_{i}^{k} \nabla q \|_{0}^{2}.
\end{equation}
\begin{lemma}\label{lem:manouzi-gen}
    Let $\Omega$ be a Lipschitz domain. Under~\Cref{assump:well-posedness-assumption}, there exists $K > 0$ such that for each $(\tilde{\mu}, \tau, p) \in \Theta^{k}$,
    \begin{equation}\label{eq:Poincare-ineq}
        \|(\tilde{\mu}, \tau, p) \|_{\Theta^k} \leq K|(\tilde{\mu}, \tau, p) |_{\Theta^{k}}.
    \end{equation}
\end{lemma}
\noindent This lemma will allow us to control the chemical potential and pressure in terms of the (linearized) driving forces, divergence of total stress, and the viscous stress. It should be thought of as the generalization to the OSM framework of~\citet[Lemma 4]{Manouzi2001}.

\begin{proof}[Proof of~\Cref{lem:manouzi-gen}]
	The first stage of the proof is to show that
\begin{equation}\label{eq:Poincare-first-inequality}
    \|p \|_0 \lesssim \|\tau\|_{0} + \| \Div \tau - \nabla p \|_{0},
\end{equation}	
following and mildly generalizing~\citet[Lemma 4]{Manouzi2001}.
Set $\theta = \tau - p\mathbb{I} - r\mathbb{I}$ where $r = \frac{1}{d | \Omega |} \int_{\Omega} \tr\tau~\dx$. Then 
\begin{equation}
    \| \tau - p\mathbb{I} \|_{0} \leq \| \theta \|_{0} + \|r \mathbb{I}\|_{0}.
\end{equation}   
As $\int_{\Omega} \tr\theta~\dx = 0$, we can use~\citet[Proposition 9.1.1]{Boffi2013} to derive
\begin{equation}
    \begin{aligned}
        \|\tau - p\mathbb{I} \|_{0} &\lesssim \|\dev\theta \|_{0} +  \| \Div \theta\|_{0} + \| r\mathbb{I} \|_{0} \\
& \lesssim \|\tau \|_{0} + \| \Div \tau - \nabla p \|_{0},
    \end{aligned}
\end{equation}
where the deviator is defined as $\dev\theta \coloneqq \theta - \frac{\tr\theta}{d}\mathbb{I} = \dev\tau$.
Now using
\begin{equation}
    \sqrt{d}\|p\|_{0} \leq \|\tau - p \mathbb{I}\|_{0} + \|\tau \|_{0},
\end{equation}
the result~\cref{eq:Poincare-first-inequality} follows.
For the second stage of the proof, we will show that
\begin{equation}\label{eq:Poincare-second-inequality}
	\|\mu_{i} \|_{0} \lesssim \|p\|_{0} + \|-c^k_i\nabla\mu_i + \omega^k_i\nabla p  \|_{0}.
\end{equation}
This combined with~\cref{eq:Poincare-first-inequality} gives~\cref{eq:Poincare-ineq}.
To prove this second inequality, for each $i$ we take the unique $z_{i} \in H^1_0(\Omega;\mathbb{R}^d)/\ker(\Div)$ such that $\Div z_i = \mu_{i}$.
Then $u_i \coloneqq z_i/c^k_i \in L^2(\Omega;\mathbb{R}^d)$ with $\Div(c^k_i u_i) = \mu_i$.
With integration by parts we deduce 
\begin{equation}
    \int_{\Omega} (-c_{i}^{k} \nabla \mu_{i} + \omega_{i}^k \nabla p) \cdot u_{i}~\dx = \int_{\Omega} |\mu_{i}|^{2} - M_ip\left( \frac{\mu_{i}}{\rho^{k}} - \frac{\nabla\rho^k}{(\rho^k)^2}\cdot c^k_i u_i\right)\dx.
\end{equation}
Upon rearrangement, we can derive the inequality
\begin{equation}
    \begin{aligned}\label{eq:Boundonmui}
        \|\mu_{i}\|^{2}_{0} &\leq M_i\|p\|_0\left(\frac{\|\mu_i\|_0}{\kappa\sum_j M_j} + \|u_i\|_0\|c^k_i\|_0\left\|\frac{\nabla\rho^k}{(\rho^k)^2}\right\|_{L^\infty(\Omega;\mathbb{R}^d)}\right) + \|-c_{i}^{k} \nabla \mu_{i} + \omega_{i} \nabla p \|_{0} \| u_{i} \|_{0}\\
        &\leq \kappa^{-1}  \|p\|_{0}\left(\|\mu_{i}\|_{0} + \|u_{i} \|_{0} \|c^k_i\|_{L^\infty(\Omega)}\|\nabla \ln \rho^{k} \|_{L^\infty(\Omega;\mathbb{R}^d)}\right) + \|-c_{i}^{k} \nabla \mu_{i} + \omega_{i} \nabla p \|_{0} \| u_{i} \|_{0}.
    \end{aligned}
\end{equation}
By the bounded inverse theorem, $\Div$ admits a bounded left inverse, so $\|z_i\|_1\lesssim \|\mu_i\|_0$ and thus
\begin{equation}
    \| u_{i} \|_{0} \leq \kappa^{-1}\|z_i\|_0 \lesssim \kappa^{-1}\|\mu_{i} \|_{0}. 
\end{equation}
Combining this with~\cref{eq:Boundonmui}, we can divide through by $\| \mu_{i} \|_{0}$ to derive
\begin{equation}\label{eq:ThefinallineofFarhloul}
    \| \mu_{i} \|_{0} \lesssim \kappa^{-1}\|p\|_{0}\left(1 + \kappa^{-1}\|c^k_i\|_{L^\infty(\Omega)}\|\nabla \ln \rho^{k} \|_{L^\infty(\Omega)}\right) + \kappa^{-1} \|-c_{i}^{k} \nabla \mu_{i} + \omega_{i}^k\nabla p \|_{0}.
\end{equation}
\end{proof}

\begin{remark}
    The constants appearing in~\cref{eq:Poincare-ineq} depend on two factors; the relative variation of the density and $\kappa$. Provided these two quantities are well-behaved across iterations, so will be the resulting constants.
\end{remark}

A further intermediate lemma we need to prove well-posedness is the following.
\begin{lemma}\label{lem:coercivity-of-A}
Under~\Cref{assump:well-posedness-assumption} there exists $\lambda_{\kappa}^{\gamma} > 0$ depending on $\kappa$ and $\gamma$ such that for all $(\tilde{v},v) \in Q$,
\begin{equation}
    \left(\frac{n+1}{2} \right)  \lambda_{\kappa}^{\gamma} \| v \|_{0}^{2} + A_{k}( \tilde{v}, v; \tilde{v}, v ) \geq \frac{ \lambda_{\kappa}^{\gamma}}{2} \|(\tilde{v}, v)\|_Q^2.
\end{equation}
\end{lemma}

\begin{proof}
Let us define $\delta_{i} = v_{i} - v$ for each $i$. Then we can explicitly compute
\begin{equation}
    A_{k}(\tilde{v},v; \tilde{v},v) = \int_{\Omega} \sum_{i,j} \delta_{i} \cdot \textbf{M}_{ij}^{k,\gamma} \delta_{j}~\dx,
\end{equation}
where $\textbf{M}^{k,\gamma}$ is defined using $\tilde{c}^k$ via~\cref{eq:augmented-transport-matrix}.
It follows from~\citet[Lemma 4.1]{Van-Brunt2021} that 
this is a coercive bilinear form in $\tilde{\delta}$: for some $\lambda_{\kappa}^{\gamma}$,
\begin{equation}
    A_{k}(\tilde{v}, v; \tilde{v}, v) \geq \frac{\lambda_{\kappa}^{\gamma}}{2} \sum_{i} \| \delta_{i} \|^{2}_{0} = \frac{\lambda_{\kappa}^{\gamma}}{2} \sum_{i} \| v_{i} -v\|^{2}_{0},
\end{equation}
hence
\begin{equation}
    \begin{aligned}
        \left(\frac{n+1}{2} \right) \lambda_{\kappa}^{\gamma} \| v \|_{0}^{2} + A_{k}(\tilde{v}, v; \tilde{v}, v) &\geq \frac{\lambda_{\kappa}^{\gamma}}{2} \sum_{i} (\| v_{i} -v\|^{2}_{0} + \| v\|^{2}_{0} ) + \frac{\lambda_{\kappa}^{\gamma}}{2} \|v\|^{2}_{0}  \\ 
        & \geq \frac{\lambda_{\kappa}^{\gamma}}{2} \sum_{i} \| v_{i} \|_{0}^{2} + \frac{\lambda_{\kappa}^{\gamma}}{2} \|v\|^{2}_{0} = \frac{  \lambda_{\kappa}^{\gamma}}{2} \|(\tilde{v}, v)\|_Q^2.
    \end{aligned}
\end{equation}
\end{proof}

We now invoke standard Babu\v{s}ka theory for well-posedness~\citep{Babuska1971}.
\begin{theorem}\label{thm:cts-well-posed}
    Under~\Cref{assump:well-posedness-assumption}, there exists a unique solution 
    to the perturbed saddle point system~\cref{eq:saddlepoint}.
\end{theorem}

\begin{proof}
For a given $(\mathfrak{p}, \mathfrak{q}) \coloneqq \left( (\tilde{\mu}, p, \tau), (\tilde{v}, v) \right) \in \Theta^{k} \times Q$ and $(\mathfrak{s}, \mathfrak{u}) \coloneqq \left( (\tilde{w}, s, q), (\tilde{u},u) \right) \in \Theta^{k} \times Q$ we will define the bounded bilinear form $\mathcal{G} : (\Theta^{k} \times Q) \times (\Theta^{k} \times Q) \rightarrow \mathbb{R}$ as  
\begin{equation}
    \mathcal{G}\left(\mathfrak{p}, \mathfrak{q}; \mathfrak{s}, \mathfrak{u} \right) \coloneqq \Lambda ( \mathfrak{p}; \mathfrak{s} ) + B_{k} (\mathfrak{s}; \mathfrak{q}) + B_{k}(\mathfrak{p}; \mathfrak{u}) - A_k(\mathfrak{q}; \mathfrak{u}).
\end{equation}
We will prove Babu\v{s}ka's inf-sup condition, namely that there exists a constant $c > 0$ such that for each $(\mathfrak{p}, \mathfrak{q}) \in \Theta^{k}\times Q$ there is $(\mathfrak{s},\mathfrak{u}) \in \Theta^{k}\times Q$ such that
\begin{equation}\label{eq:babuska-condition}
    \frac{\mathcal{G}\left(\mathfrak{p}, \mathfrak{q}; \mathfrak{s}, \mathfrak{u}\right)}{\|(\mathfrak{s}, \mathfrak{u}) \|_{\Theta^{k} \times Q}} \geq c \|(\mathfrak{p}, \mathfrak{q} )\|_{\Theta^{k} \times Q},
\end{equation}
with product norm $\| (\mathfrak{p}, \mathfrak{q}) \|_{\Theta^{k} \times Q}^{2} \coloneqq \| \mathfrak{p}\|_{\Theta^{k}}^{2} + \| \mathfrak{q}\|_{Q}^{2}$.
Note that $\mathcal{G}$ is defined on the product of a space with itself and is symmetric, and so only the one inf-sup condition~\cref{eq:babuska-condition} needs to be verified. 
Proving~\cref{eq:babuska-condition} will be accomplished by showing that for a constant $c > 0$, for each $(\mathfrak{p}, \mathfrak{q}) \in \Theta^{k} \times Q$ there is $(\mathfrak{s}, \mathfrak{u}) \in \Theta^{k} \times Q$ such that 
$
    \mathcal{G}\left( \mathfrak{p}, \mathfrak{q}; \mathfrak{s}, \mathfrak{u} \right) \geq c \| (\mathfrak{p}, \mathfrak{q}) \|_{\Theta^{k} \times Q}^{2},
$
and for a $C > 0$ independent of $(\mathfrak{p},\mathfrak{q})$,
\begin{equation}\label{eq:bound-on-r-k}
    C \| (\mathfrak{p}, \mathfrak{q} ) \|_{\Theta^{k} \times Q} \geq  \| (\mathfrak{s}, \mathfrak{u} ) \|_{\Theta^{k} \times Q}.
\end{equation}
This combined with our Poincar\'e-type inequality will imply well-posedness. 
We do this by fixing $(\mathfrak{s}, \mathfrak{u})$ as the ansatz
\begin{equation}\label{eq:test-function-choices}
    \begin{aligned}
        &w_{i} = C_{1} \mu_{i}, \qquad s = C_1 \tau + C_2 s_v, \qquad &&q = C_{1}p, \\
        &u_i = C_3 (-c^{k}_{i} \nabla \mu_{i} + \omega^{k}_{i} \nabla p) - C_1 v_i, \qquad &&u = -C_1 v + C_4 (\Div \tau - \nabla p).
    \end{aligned}
\end{equation}
Here $C_1, \ldots > 0$ are constants to be set later and $s_v \in \hdivs$ is chosen to satisfy $\Div s_v = v$ and $\|s_{v}\|_{\hdivs} \leq C_{\Sigma} \|v\|_{0}$ for a constant $C_{\Sigma}$ independent of $v$ (its existence is guaranteed by~\citet[Theorem 8.1]{Arnold2018}). It is clear that~\cref{eq:bound-on-r-k} holds. 
With these choices of test functions we may compute
\begin{equation}\label{eq:computed-result}
    \begin{aligned}    
        \mathcal{G}\left( \mathfrak{p}, \mathfrak{q}; \mathfrak{s}, \mathfrak{u} \right) = \int_{\Omega}& \mathcal{A} \tau : (C_1 \tau + C_2 s_{v})~\dx + C_3 \sum_{i} \| - c^{k}_{i} \nabla \mu_{i} + \omega^{k}_{i} \nabla p \|_{0}^{2} \\
        &+ C_4 \| \Div \tau - \nabla p \|^{2}_{0} + C_{2} \|v\|^{2}_{0}   - A_{k} \left( \tilde{v}, v; \tilde{u}, u \right),
    \end{aligned}
\end{equation}
and observe that the final term on the right may be written equivalently as
\begin{equation}\label{eq:last-term}
    C_{1} A_{k} \left( \tilde{v}, v; \tilde{v}, v \right) - A_{k} \left( \tilde{v}, v; \{C_3(-c_{i}^{k} \nabla \mu_{i} + \omega_{i}^{k} \nabla p)\}_i, C_4(\Div \tau - \nabla p) \right). 
\end{equation}
With $\lambda_{\kappa}^{\gamma}$ given by~\cref{lem:coercivity-of-A} we now choose $C_{2} = \|A_{k}\|^{2}(n+1) \lambda_{\kappa}^{\gamma}$ and assume that $C_{1} \geq 2 \|A_{k} \|^{2}$. Then 
\begin{equation}
    C_2\|v\|^{2}_{0} +  C_{1} A_{k} \left( \tilde{v},v; \tilde{v}, v \right) \geq  \|A_{k}\|^{2}  \lambda_{\kappa}^{\gamma} \| (\tilde{v}, v)\|_Q^2. 
\end{equation}
Using uniform positive definiteness of the compliance tensor, the bound on $s_{v}$, and boundedness of the operators $\Lambda, A_{k}$, we can proceed to bound~\cref{eq:computed-result}
from below as
\begin{equation}
    \begin{aligned}
        \mathcal{G}\left( \mathfrak{p}, \mathfrak{q}; \mathfrak{s}, \mathfrak{u} \right) &\geq \alpha C_{1} \| \tau \|^{2}_{0} +   \|A_{k}\|^{2} \lambda_{\kappa}^{\gamma} \| (\tilde{v}, v)\|_Q^2 + C_{3} \sum_{i} \| -c_{i}^{k} \nabla \mu_{i} + \omega_{i}^{k} \nabla p \|_{0}^{2}  \\ 
        & + C_{4} \| \Div \tau - \nabla p \|_{0}^{2} - \lambda_{\kappa}^{\gamma} (n+1)C_{\Sigma} \|A_{k} \|^{2}  \| \Lambda \| \|v \|_{0} \| \tau \|_{0} \\ 
        &  - \| A_{k} \| \| (\tilde{v}, v)\|_Q \left( C^{2}_{3}\sum_i \| -c_{i}^{k} \nabla \mu_{i} + \omega_{i}^{k} \nabla p \|^{2}_{0} + C^{2}_{4}\| \Div \tau - \nabla p \|^{2}_{0} \right)^{1/2}.
    \end{aligned}
\end{equation}
Here $\alpha > 0$ is such that $\int_{\Omega} \mathcal{A}\tau : \tau~\dx \geq \alpha \| \tau\|_{0}^{2}$ for all $\tau\in L^2(\Omega;\mathbb{S})$.
The desired inequality may now be derived by judiciously selecting the constants $C_1, C_3, C_4$ (typically by choosing $C_{1} \gg C_{2} \gg \text{max} (C_{3},C_{4})$) and using the weighted Young inequality. For concreteness, constants we might pick are
\begin{equation}
    \begin{aligned}
        C_1 = \left(\frac{\lambda_{\kappa}^{\gamma} C^{2}_{\Sigma} \| \Lambda\|^{2} (n+1)^{2} }{\alpha} + 2\right)\| A_{k}\|^{2}, \qquad C_{3} = C_{4} =  \lambda_{\kappa}^{\gamma}.
    \end{aligned}
\end{equation}
With this choice and our Poincar\'e-type inequality from~\Cref{lem:manouzi-gen}, combined with the inequality~\cref{eq:bound-on-r-k} we may derive
\begin{equation}
    \begin{aligned}
        \mathcal{G}\left( \mathfrak{p}, \mathfrak{q}; \mathfrak{s}, \mathfrak{u} \right) &\geq 2\alpha\|A_k\|^2\|\tau\|_{0}^{2} + \frac{\lambda_{\kappa}^{\gamma}}{6}\|A_k\|^2 \| (\tilde{v}, v)\|_Q^2 \\
        &+ \frac{ \lambda_{\kappa}^{\gamma}}{4} \left( \sum_i \| -c_{i}^{k} \nabla \mu_{i} + \omega_{i}^{k} \nabla p \|_{0}^{2}  + \| \Div \tau - \nabla p \|^{2}_{0} \right) \\
        & \gtrsim | \mathfrak{p} |_{\Theta^{k}}^{2} + \|\mathfrak{q}\|_Q^2 \gtrsim  \| (\mathfrak{p}, \mathfrak{q}) \|_{\Theta^{k} \times Q}^{2} \gtrsim  \| (\mathfrak{p}, \mathfrak{q}) \|_{\Theta^{k} \times Q} \| (\mathfrak{s}, \mathfrak{u}) \|_{\Theta^{k} \times Q},
    \end{aligned}
\end{equation}
which is the statement of the Babu\v{s}ka condition~\cref{eq:babuska-condition}.
\end{proof}

\section{Discretization and numerical experiments}\label{sec:discretization}

We now assume that $\Omega$ is polytopal, and
admits a quasi-uniform triangulation
$\mathcal{T}_h$ 
with simplicial elements of maximal diameter $h$.
Denote conforming finite element spaces for the discrete solution tuple by
\begin{equation}
    \begin{aligned}
        \underbrace{(X_h^n\times \Sigma_h\times P_h)}_{\eqqcolon~\Theta^k_h}\times\underbrace{(W_h^n\times V_h)}_{\eqqcolon~Q_h}
        \subset \underbrace{(L^2_0(\Omega)^n \times L^2(\Omega;\mathbb{S})\times L^2_0(\Omega))}_{\supset~\Theta^k}\times \underbrace{(L^{2}(\Omega;\mathbb{R}^d)^n\times L^2(\Omega; \mathbb{R}^d))}_{=~Q}.
    \end{aligned}
\end{equation}
Here $\Theta^k_h$ is independent of $k$ as a set, but
inherits an iteration-dependent norm described below; 
$Q_h$ inherits the norm of $Q$.
Our discretized linear problem after $k\geq 0$ nonlinear iterations therefore reads:
seek $((\tilde{\mu}_h, \tau_h, p_h), (\tilde{v}_h, v_h)) \in \Theta^k_h\times Q_h$ such that
\begin{equation}\label{eq:discrete-saddle-point}
    \begin{aligned}
        \Lambda(\tilde{\mu}_{h}, \tau_{h}, p_{h}; \tilde{w}_{h}, s_{h}, q_{h}) + B_{k,h}(\tilde{w}_{h}, s_{h}, q_{h}; \tilde{v}_{h}, v_{h}) &= \ell^{1}_{k,h}(\tilde{w}_{h}, s_{h}, q_h)\hspace{-3mm}&&\forall(\tilde{w}_h, s_h, q_h)\in \Theta^k_{h}, \\
        B_{k,h}(\tilde{\mu}_{h}, \tau_{h}, p_{h}; \tilde{u}_{h}, u_{h}) - A_{k,h}(\tilde{v}_{h}, v_{h}; \tilde{u}_{h}, u_{h}) &= \ell^{2}_{k,h}(\tilde{u}_h, u_{h})&&\forall(\tilde{u}_h, u_h)\in Q_h,
    \end{aligned}
\end{equation}
where $A_{k,h}, B_{k,h}$ are obtained from $A_k, B_k$, and $\ell^1_{k,h}, \ell^2_{k,h}$ from $\ell^1_k, \ell^2_k$, respectively, by replacing the discretely computed concentrations $c^k_{i}$ and inverse density $(\rho^k)^{-1}$ with discrete approximations;
we use these to define a norm $\|\cdot\|_{\Theta^k_h}$ for $\Theta^k_h$ in analogy to~\cref{eq:theta-norm}. 
In block form, the linearized discrete problem reads
\begin{equation}\label{eq:discrete-linearized}
        \nicediscretemat
            \Block{3-3}{\Lambda} & & & \Block{3-2}{B_{k,h}^\top} & \\
            & & & &\\
            & & & &\\
            \hline
            \Block{2-3}{B_{k,h}} & & & \Block{2-2}{~~-A_{k,h}}&\\
            & & & &
        \end{bNiceArray}
        \nicevec
                \tilde{\mu}_h \\ \tau_h \\ p_h
            \\
            \hline
                \tilde{v}_h \\ v_h
        \end{bNiceArray}
        =
        \nicevec
            \Block{3-1}{
                \ell^1_{k,h}
            }\\
            \\
            \\
            \hline
            \Block{2-1}{
                \ell^2_{k,h}
            }\\
            \\ 
    \end{bNiceArray}.
\end{equation}

\subsection{Structure-preservation and well-posedness}\label{sec:structural-assump}

We have already 
argued 
the need for pressure regularity greater than $L^{2}$. We therefore employ
the continuous Lagrange element of degree $t\geq 1$, $P_h = \text{CG}_t(\mathcal{T}_h)$, for the pressure. 
From the diffusion driving forces~\cref{eq:driving-force}, it appears natural to take the chemical potential space $X_h$ to be CG elements of at least the same degree $r \geq t$, $X_h = \text{CG}_{r}(\mathcal{T}_h)$.

The mass-average velocity constraint~\cref{eq:m-a-vel-constraint} suggests that the species velocity space be contained in the space used for convective velocity, $W_h\subset V_h$.
For the Stokes subsystem, it is desirable that the pair $(\Sigma_h\times P_h, V_h)$ be inf-sup compatible, for which it is sufficient to have that the full divergence $(\tau, p)\mapsto\Div \tau - \nabla p$ is surjective onto $V_{h}$.
By the regularity choice~\cref{eq:stress-pressure-space} for the pressure, it is thus natural to apply $\Div$-conforming tensor elements to discretize the viscous stress.
By the decomposition~\cref{eq:low-Re-cauchy-stress}, symmetry of the viscous stress is equivalent to the conservation of angular momentum;
for now, we consider exactly symmetric stress elements (such as those proposed in~\citet{Arnold2002} and~\citet{Arnold2008}) since this obviates the need for a symmetry-enforcing Lagrange multiplier
which would add a further field to our $(2n + 3)$-field formulation.

However, if one would like to repeat at the discrete level the proof of~\Cref{thm:cts-well-posed}, it is necessary for $\Div:\Sigma_h\to V_h$ to be surjective, allowing us to construct the discrete analogue of the tensor field $s_v$ in the ansatz~\cref{eq:test-function-choices}. This is stronger than surjectivity of $(\tau, p)\mapsto \Div\tau - \nabla p, \Sigma_h\times P_h\to V_h$,
but in practice is equivalent because many appropriate choices of $\Sigma_h$ are designed to discretize the \textit{full} Cauchy stress.
Furthermore,
the discrete analogue of the constant $C_\Sigma$ (and hence the resulting inf-sup constant) will \textit{a priori} depend on $h$;
it is therefore necessary to assume that such $s_v$ can be constructed stably.

\begin{assumption}\label{assump:stress-assump}
    There exists $C^\Sigma$ independent of $h$ such that for each $u_h\in V_h$ and for the unique $s_h\in\Sigma_h/\ker(\Div)$ with $\Div s_h = u_h$, there holds $\|s_h\|_{H(\Div;\mathbb{S})}\leq C^\Sigma\|u_h\|_0$.
\end{assumption}
This is true for (for example) stress elements discretizing an elasticity complex which admits bounded commuting projections to the subcomplex, as is for example the case for the Arnold--Winther and Arnold--Awanou--Winther elements.
The other relations are summarised below,\footnote{Here $\pi_i$ denotes projection onto the $i^{\text{th}}$ component, i.e.~we require $P_h\subset X_h$.}
\begin{equation}\label{eq:discrete-complex}
    \begin{tikzcd}[column sep=huge]
        \underset{\substack{\text{\scriptsize chemical} \\ \text{\scriptsize potential}}}{X_h} & \underset{\substack{\text{\scriptsize species} \\ \text{\scriptsize velocity}}}{W_h} \arrow{d}{\subset} \\
        \underset{\scriptsize\text{stress}\times\text{pressure}}{\Sigma_h\times P_h} \arrow{u}{\pi_2} \arrow[twoheadrightarrow]{r}{(\tau, p)\mapsto\Div\tau} &[20em] \underset{\substack{\text{\scriptsize convective} \\ \text{\scriptsize velocity}}}{V_h}
    \end{tikzcd}
\end{equation}
where $\twoheadrightarrow$ indicates surjectivity. 
The bottom row corresponds to the 
final 
segment of a discrete stress elasticity complex, with stress space refined for Stokes flow due to the decomposition~\cref{eq:low-Re-cauchy-stress}.
We will need the conditions of~\Cref{lem:manouzi-gen} to hold exactly in the discretization. This will in general require that we approximate the concentrations, $c^{k}_{i}$, and \textit{density reciprocal}, $(\rho^k)^{-1}$, in specific discrete function spaces. The interpolation of these terms will be denoted by $c^{k}_{i,h}$ and $\rho^{k,-1}_{h}$, respectively. 

Finally, to show well-posedness of the discrete problem, we require an additional condition which does not fit neatly onto~\cref{eq:discrete-complex}.
\begin{assumption}\label{assump:discrete-assumption}
    The operator given by
    \begin{equation}\label{eq:discrete-driving force}
        \mathrm{d}_h^{i,k}(w_h, q_h) \coloneqq -c_{i,h}^k\nabla w_{h} + \omega^k_{i,h}\nabla q_h,
    \end{equation}
    where
    $\omega^{k}_{i,h} \coloneqq M_{i} c^{k}_{i,h} \rho^{k,-1}_{h}$,
    is 
    well-defined as a map $\mathrm{d}_h^{i,k}:X_h\times P_h\to W_h$, i.e., it takes values in $W_h$.
\end{assumption}

\begin{remark}\label{rem:weaker-structural-cdns}
    Note that~\Cref{lem:manouzi-gen} required the gradient of $\rho^{-1}$, and so $\rho^{k,-1}_{h}$ should at least be a continuous piecewise linear function. In order to minimise the polynomial degree for $W_h$ arising from~\Cref{assump:discrete-assumption}, it is advantageous to interpolate $c^k_{i,h}$ onto the space $\text{DG}_{0}$. These coefficients do not affect the accuracy of the discretization, only the quality of the linearization, and nonlinear convergence appears robust regardless of this approximation.
\end{remark}

\begin{theorem}
    Under Assumptions~\ref{assump:stress-assump} and~\ref{assump:discrete-assumption} and the relations specified in~\cref{eq:discrete-complex}, the
    the system~\cref{eq:discrete-saddle-point} is well-posed, uniformly in $h$.
\end{theorem}

\begin{proof}
Due to the structural conditions demanded in the Assumptions, by inspection the choices of test functions~\cref{eq:test-function-choices} are valid. As a consequence we may completely replicate the argument presented in the infinite-dimensional case, and derive the analogue of condition~\cref{eq:babuska-condition} with constant independent of $h$.
\end{proof}

\subsection{Error estimates}

Following~\citet[Theorem 2]{Xu2003}, for fixed $k$ we infer the abstract error estimate
\begin{equation}\label{eq:abstract-error}
    \|(\underline{\tilde{\mu}}^{k+1}, \underline{\tau}^{k+1}, \underline{p}^{k+1}) - (\tilde{\mu}_h, \tau_h, p_h)\|_{\Theta^k_h} + \|(\underline{\tilde{v}}^{k+1}, \underline{v}^{k+1}) - (\tilde{v}_h, v_h)\|_Q\lesssim \mathbb{E}_{\Theta^k_h} + \mathbb{E}_{Q_h},
\end{equation}
where
\begin{equation}
    \begin{aligned}
    \mathbb{E}_{\Theta^k_h} &\coloneqq \inf_{(\tilde{w}_h,s_h,q_h)\in\Theta^k_h}\|(\underline{\tilde{\mu}}^{k+1}, \underline{\tau}^{k+1}, \underline{p}^{k+1}) - (\tilde{w}_h,s_h,q_h)\|_{\Theta^k_h},\\
    \mathbb{E}_{Q_h} &\coloneqq \inf_{(\underline{\tilde{u}}_h, u_h)\in Q_h}\|(\underline{\tilde{v}}^{k+1}, \underline{v}^{k+1}) - (\tilde{u}_h, u_h)\|_Q.
    \end{aligned}
\end{equation}
Here the tuple $((\underline{\tilde{\mu}}^{k+1}, \underline{\tau}^{k+1}, \underline{p}^{k+1}), (\underline{\tilde{v}}^{k+1}, \underline{v}^{k+1}))$ is defined as the exact solution to~\cref{eq:saddlepoint} but with $B_{k}, A_{k}, \ell^1_k, \ell^2_k$ replaced with $B_{k, h}, A_{k, h}, \ell^1_{k, h}, \ell^2_{k, h}$, respectively---that is, the solution of the system~\cref{eq:saddlepoint} with the estimates of the concentrations and inverse density replaced by $c_{i,h}^{k}$ and $\rho^{k,-1}_{h}$, respectively.

To derive a practical error estimate, we will need to bound the quantities $\mathbb{E}_{\Theta^{k}_h}$ and $\mathbb{E}_{Q_h}$ 
by interpolation estimates specific to the choice of finite element spaces, by estimating $\|\cdot\|_{\Theta^k_h}, \|\cdot\|_Q$ in terms of standard Sobolev norms. 
To this end we can readily check that
\begin{equation}
    \begin{aligned}
        \mathbb{E}_{\Theta^{k}_h}&\lesssim \max\left(1, \sum_i \|c^k_{i,h} \|_{L^{\infty}(\Omega)}\right)\inf_{\tilde{w}_{h} \in X_{h}^n} \| \underline{\tilde{\mu}}^{k+1} - \tilde{w}_{h} \|_{1} \\
        &+ \max\left(1, \sum_i \|\omega^{k}_{i,h}\|_{L^\infty(\Omega)}\right) \inf_{q_{h} \in P_{h}} \| \underline{p}^{k+1} - q_{h} \|_{1} + \inf_{s_{h} \in \Sigma_{h}} \|\underline{\tau}^{k+1} - s_{h} \|_{H(\Div;\mathbb{S})}, \\ 
        \mathbb{E}_{Q_h} &\lesssim \inf_{\tilde{u}_{h} \in W_{h}^n} \| \underline{\tilde{v}}^{k+1} - \tilde{u}_{h} \|_{0} + \inf_{u_{h} \in V_{h} } \| \underline{v}^{k+1} - u_{h} \|_{0}.
    \end{aligned}
\end{equation}

\subsection{Examples of suitable finite elements}

Having derived abstract error estimates, we now consider two simple combinations of finite elements satisfying the structural conditions~\cref{eq:discrete-complex} and Assumptions~\ref{assump:stress-assump} and~\ref{assump:discrete-assumption}.

The design and implementation of stress elements which exactly enforce symmetry and $\Div$-conformity is notoriously difficult; in 2D, one choice of such elements is the conforming
Arnold--Winther element~\citep{Arnold2002}, 
recently incorporated into the Firedrake finite element library~\citep{Rathgeber2016, Aznaran2021}.
In the lowest-order case we denote this element by $\text{AW}^c_3$.
Specifying
\begin{subequations}\label{eq:canonical-FE-family}
    \begin{alignat}{2}
        X_h = P_h &= \text{CG}_1(\mathcal{T}_h)\cap L^2_0(\Omega),\label{eq:potential-pressure-element}\\
        \Sigma_h &= \text{AW}^c_3(\mathcal{T}_h),\\
        W_h = V_h &= \text{DG}_1(\mathcal{T}_h;\mathbb{R}^d),
    \end{alignat}
\end{subequations}
and further assuming that the discretely computed $c^k_{i}$ and $(\rho^k)^{-1}$ have been interpolated into $\text{DG}_{0}$ and $\text{CG}_{1}$, respectively, then this discretization satisfies the structural properties~\cref{eq:discrete-complex} and Assumptions~\ref{assump:stress-assump} and~\ref{assump:discrete-assumption}, hence we deduce the error estimate~\cref{eq:abstract-error}.

Let $\Pi_{h}^{\text{CG}_{1}}:H^{2}(\Omega)\to\text{CG}^{1}(\mathcal{T}_{h}), \Pi^{\text{AW}^c_3}_{h}:H^1(\Omega;\mathbb{S})\to\text{AW}^c_3(\mathcal{T}_h)$, and $\Pi_{h}^{\text{DG}_1^d}:H^1(\Omega;\mathbb{R}^d)\to\text{DG}_1(\mathcal{T}_{h};\mathbb{R}^d)$ be the associated interpolation operators. We then have the following estimates under sufficient regularity assumptions (for details we refer to~\citet{Arnold2002},~\citet[p.~72]{Boffi2013} and~\citet[Ch.~3]{Logg2012}): 
\begin{subequations}
    \begin{alignat}{2}
        &\|\tilde{\mu} - \Pi^{\text{CG}_1}_{h} \tilde{\mu}\|_{1} \lesssim h |\tilde{\mu} |_{2}, \\
        &\| p - \Pi_h^{\text{CG}_1}p \|_{1} \lesssim h |p|_{2},\\
        &\| \tau - \Pi^{\text{AW}^c_3}_{h} \tau \|_{0} + h \| \Div (\tau - \Pi_{h}^{\text{AW}^c_3} \tau) \|_0 \lesssim h^2 |\tau|_2, \\
        &\| (\tilde{v}, v) - \Pi^{\text{DG}_1^d}_{h}(\tilde{v}, v) \|_{Q} \lesssim h^{2} | (\tilde{v}, v) |_{1},
    \end{alignat}
\end{subequations}
where $\Pi^{\text{CG}_1}_h, \Pi^{\text{DG}_1^d}_h$ have been applied component-wise.
Using these estimates for the interpolation operators and the error estimate~\cref{eq:abstract-error}, we can derive the error bound
\begin{equation}\label{eq:Errorestimateforcanonical}
    \|(\underline{\tilde{\mu}}^{k+1}, \underline{\tau}^{k+1}, \underline{p}^{k+1}) - (\tilde{\mu}_{h}, \tau_{h}, p_{h}) \|_{\Theta^{k}_h} + \| (\underline{\tilde{v}}^{k+1}, \underline{v}^{k+1}) - (\tilde{v}_{h}, v_{h}) \|_{Q} \lesssim h.
\end{equation}
We will see in practice that the order of convergence for several fields is actually higher, but the error of the species velocities and the driving forces is ${\rm O} (h)$.

A second class of finite elements may be found by replacing~\cref{eq:potential-pressure-element} with
\begin{equation}\label{eq:second-finite-element}
    X_h = \text{CG}_2(\mathcal{T}_h)\cap L^2_0(\Omega), P_h = \text{CG}_1(\mathcal{T}_h)\cap L^2_0(\Omega),
\end{equation}
and leaving the others unchanged.
Again the structural conditions are satisfied if $c^k_{i}$ and $(\rho^k)^{-1}$ are interpolated into $\text{DG}_{0}$ and $\text{CG}_{1}$, respectively. A similar error analysis again confers an error bound of ${\rm O} (h)$, though shortly we will see that this is actually higher in practice. 

\begin{remark}
    These estimates bound the error between the discrete solutions at iteration $k + 1, ((\tilde{\mu}_{h}, \tau_{h}, p_{h}), (\tilde{v}_{h}, v_{h}))$ and the continuous solution of the nonlinear scheme $((\underline{\tilde{\mu}}^{k+1}, \underline{\tau}^{k+1}, \underline{p}^{k+1}), (\underline{\tilde{v}}^{k+1}, \underline{v}^{k+1}))$ with the same (discrete) coefficients.
In principle this is incomplete, as ideally we would derive error estimates against the continuous solution $((\tilde{\mu}^{k+1}, \tau^{k+1}, p^{k+1}), (\tilde{v}^{k+1}, v^{k+1}))$ at iteration $k + 1$ with the exact (continuous) coefficients. Estimates on such consistency errors were analysed for a simpler system in~\cite{Van-Brunt2021} and some rationale was provided as to why in practice this is not an issue, based on the formal order of the consistency error being strictly less than the discretization error. We expect a similar (if laborious) analysis could be performed, following the strategy in~\cite{Van-Brunt2021}. In general the consistency errors are expected to be ${\rm O}(h^{2})$, which will be borne out in the numerical examples.
\end{remark}

\begin{remark}
    We emphasize that we have aimed to identify appropriate structural conditions between finite element spaces in order to preserve desirable properties of the SOSM system---in particular, conditions which allow mimicry of well-posedness proofs from the infinite-dimensional problem---rather than to advocate specifically for the elements used here. We expect it is possible to use Lagrange multipliers to weakly enforce the symmetry of the viscous stress, which would allow for the choice of higher polynomial degrees.
\end{remark}

The system matrix of our discrete linearized system~\cref{eq:discrete-linearized} has symmetric perturbed saddle point structure, and although indefinite, is such that both the blocks $\Lambda, A_{k,h}$ are symmetric positive semidefinite. These matrix properties hold independently of the particular material considered. We expect that this structure should be conducive to the development of fast preconditioners.

\subsection{Validation with manufactured solutions}\label{sec:numerics}

We now verify our scheme, implemented in Firedrake~\citep{Rathgeber2016}.
Firedrake currently only supports symmetry-enforcing stress elements in 2D, and we thus restrict ourselves to the case $d=2$.
Throughout these experiments we chose the penalty parameter $\gamma = 0.1$, and 
the linear systems were solved using the sparse LU factorization of MUMPS~\citep{Amestoy2001} via PETSc~\citep{Balay2019}.

To validate our error estimates, we construct a manufactured solution for an ideal gas on the unit square $\Omega = (0,1)^2$. If $RT = 1$, the diffusion coefficients are given by $\mathscr{D}_{ij} = D_{i}D_{j}$ for $D_{j} > 0$, and $g:\mathbb{R}^{2}\to\mathbb{R}$ is smooth, then one can check 
that an analytical solution to the OSM subsystem~\cref{eq:OSM-eqn} is given by
\begin{equation}
    c_{i} = \exp\left(\frac{g}{D_{i}}\right), \quad v_i = D_i\nabla g,
\end{equation}
which together implicitly define every other quantity (for given shear and bulk viscosities) apart from the chemical potentials. We compute the latter by inverting the ideal gas relation~\cref{eq:ideal-gas-chemical-potential}
with $p^\ominus = \dashint_\Omega p~\dx, \mu_i^\ominus = \dashint_\Omega c_i~\dx~\forall i$.
The molar mass of each species was set to $1$, and $\zeta, \eta$ to $0.1$. The initial guesses for the concentrations $\tilde{c}^0$  were set as $c^{0}_{i} = \dashint_\Omega c_{i}~\dx$, i.e.~as the average of the exact concentrations.

We used $D_i = \frac{1}{2} + \frac{i}{4}, i = 1, 2, 3$, and $g(x,y) = \frac{xy}{5}(1 - x)(1 - y)$ to generate~\Cref{fig:mms-conv}, the log-log error plot for the overall algorithm, which demonstrates the negligible effect of the consistency error induced by the discrete interpolations $c^k_{i,h}, \rho^{k,-1}_h$, and verifies the error estimate~\cref{eq:Errorestimateforcanonical}.

\begin{figure}
    \begin{subfigure}[t]{.49\textwidth}  
        \centering
        \includegraphics[width=1\linewidth]{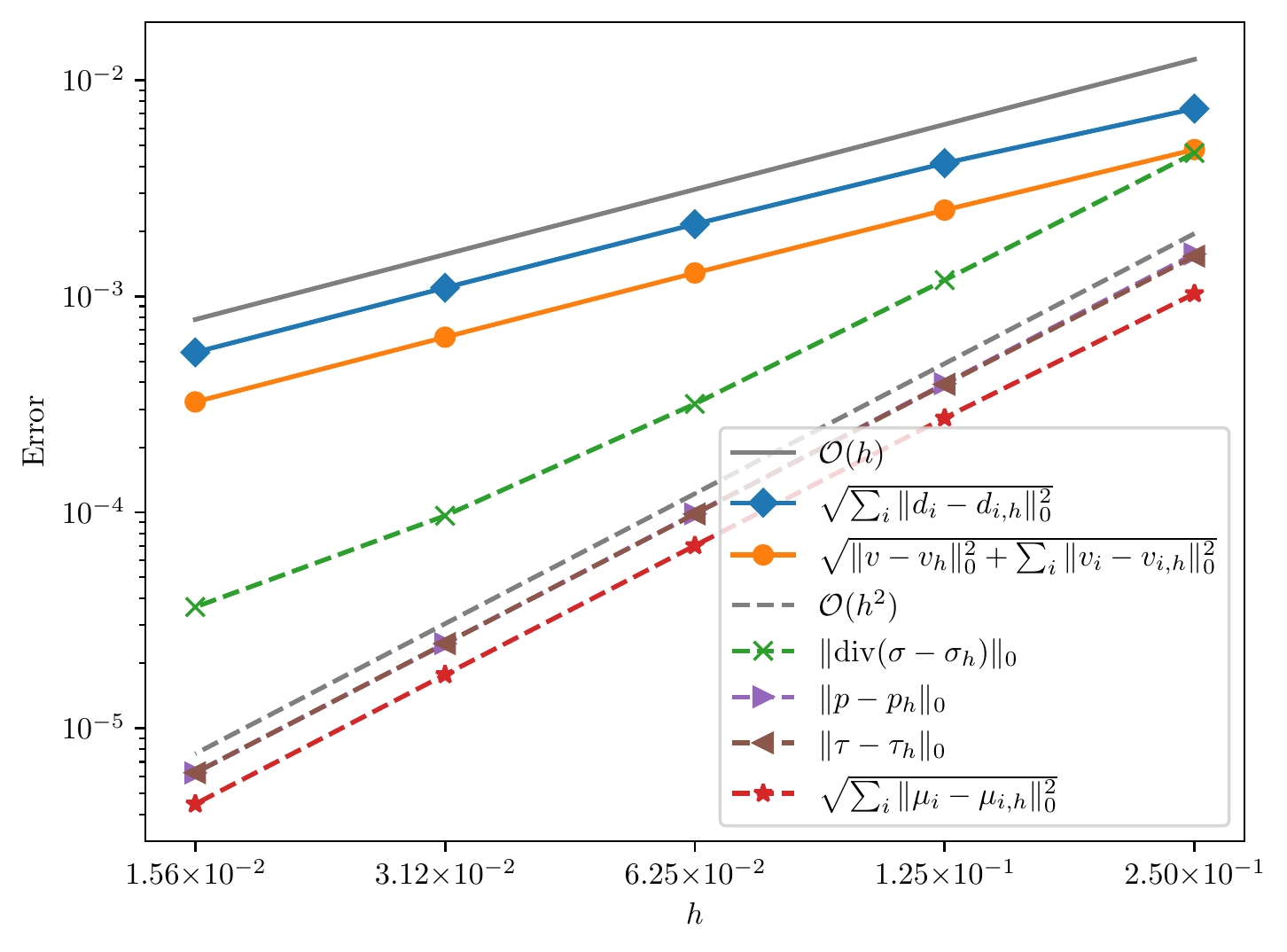}
    \end{subfigure}
    \begin{subfigure}[t]{.49\textwidth} 
        \centering
        \includegraphics[width=1\linewidth]{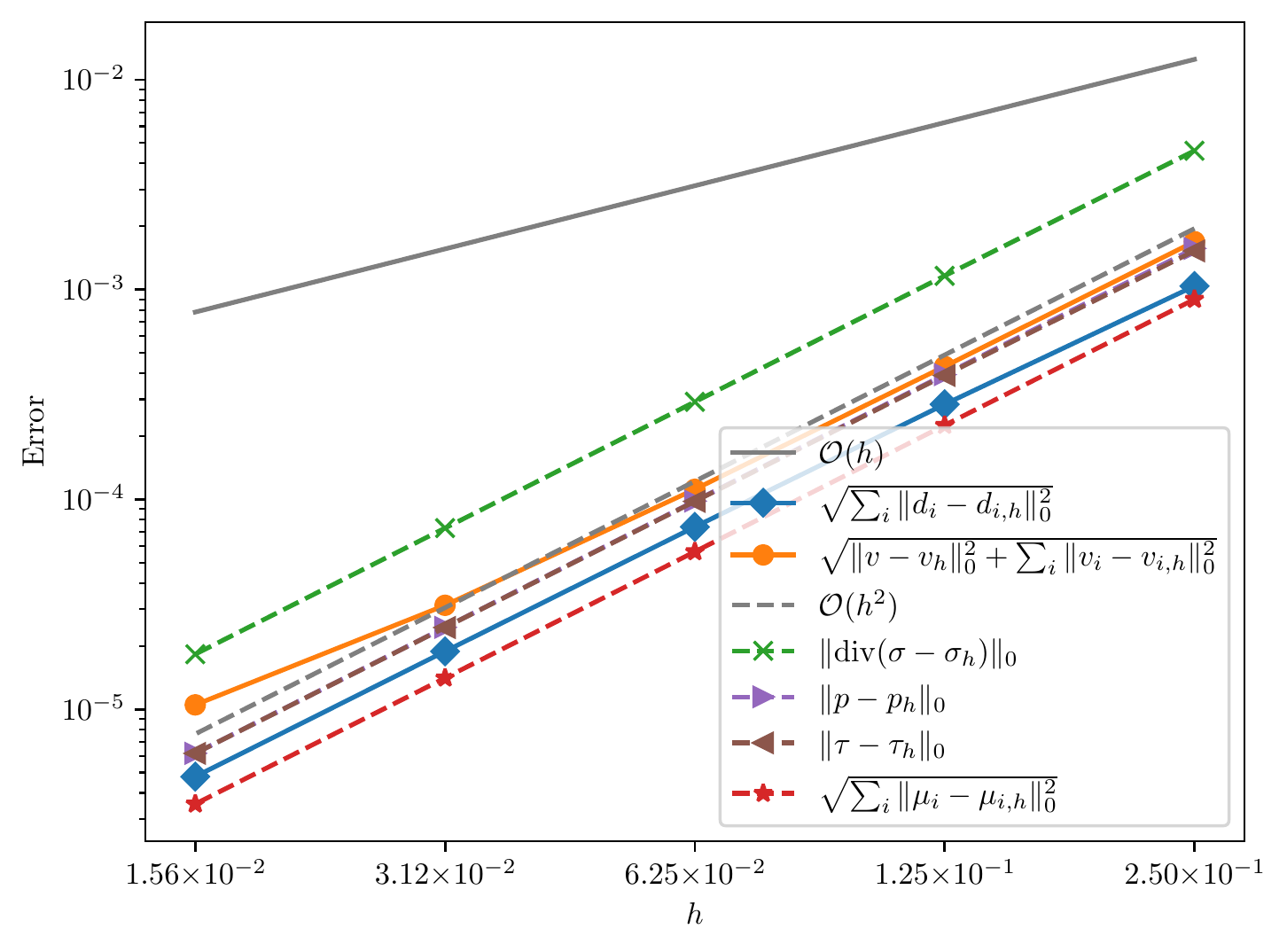}
    \end{subfigure}
    \caption{Error plots for two finite element families:~\cref{eq:canonical-FE-family} (left),
    and~\cref{eq:second-finite-element} (right).
    }\label{fig:mms-conv}
\end{figure}
The tolerance in the outer solver was $10^{-7}$ in the $\ell^2$ norm, and took $6$ nonlinear iterations on the coarsest mesh of $4 \times 4$, to $7$ iterations on finest mesh of $32 \times 32$. We have denoted in~\cref{fig:mms-conv} $d_{i,h}$ as the discrete driving force defined by~\cref{eq:discrete-driving force} at the final iteration, and $\sigma_{h} \coloneqq \tau_{h} - p_{h} \mathbb{I}$.
Note that we observe ${\rm O}(h^{2})$ convergence in the $L^{2}$ norms of the chemical potential, stress, and pressure. As in~\cite{Van-Brunt2021}, this suggests that the error estimates could be improved, for example by duality arguments.  

Due to our construction of the `linearized' function spaces~\cref{eq:theta-space}, it is the norm $\|\cdot\|_{\Theta^k_h}$ with respect to which we have proved convergence of the solution tuple. It is natural to wonder whether this is an artefact of our constructed function spaces. 
To answer this, we measure convergence of the chemical potential gradients $\nabla\mu_i$, pressure gradient $\nabla p$, and divergence $\Div\tau$ of the non-equilibrium stress to their true values, compared to the convergence of the nonlinear diffusion driving forces and the divergence of the full Cauchy stress.
For the former quantities, there is \textit{a priori} no reason to expect any convergence.

\begin{figure}
    \begin{subfigure}[t]{.49\textwidth}  
        \centering
        \includegraphics[width=1\linewidth]{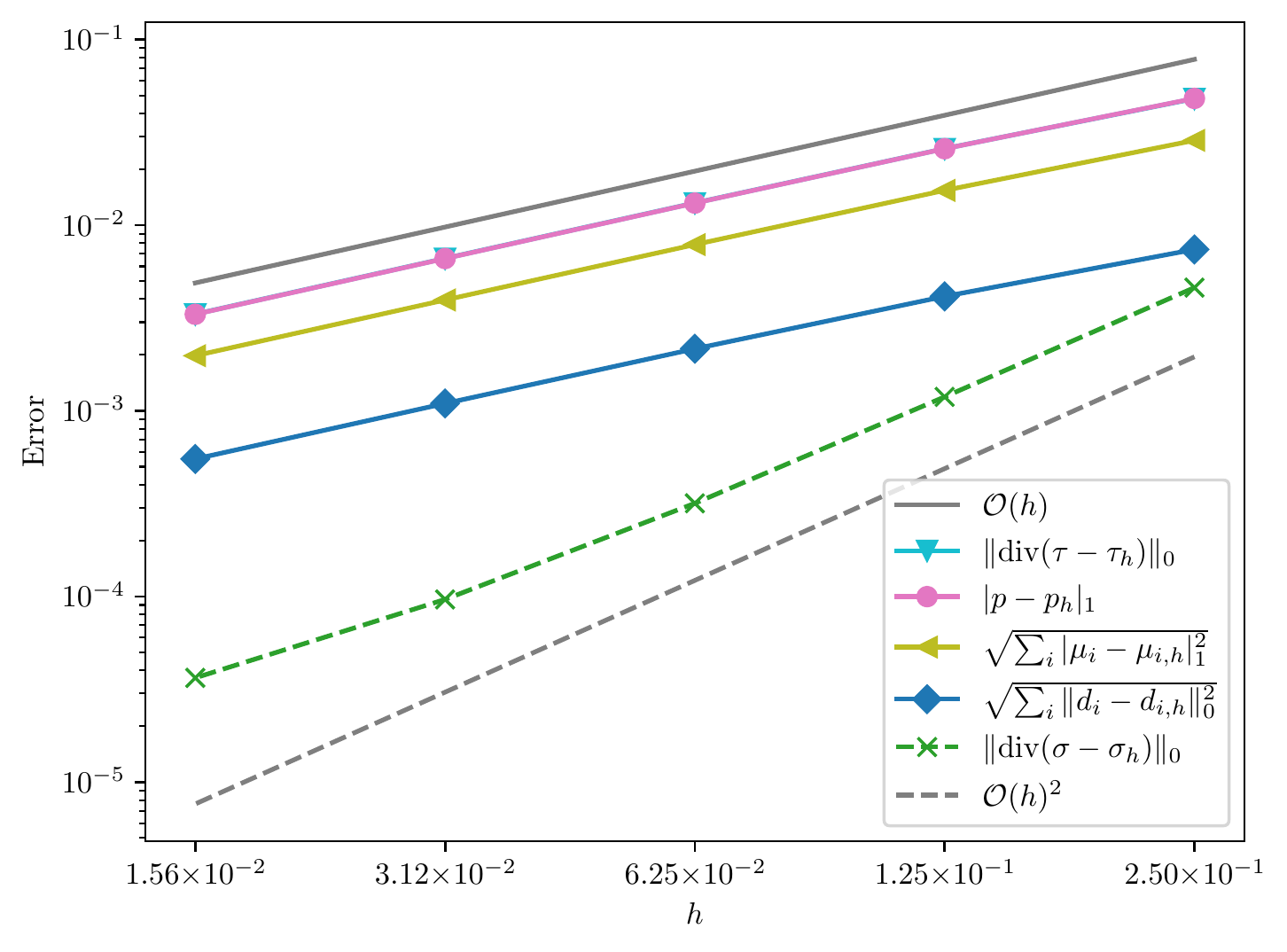}
    \end{subfigure}
    \begin{subfigure}[t]{.49\textwidth} 
        \centering
        \includegraphics[width=1\linewidth]{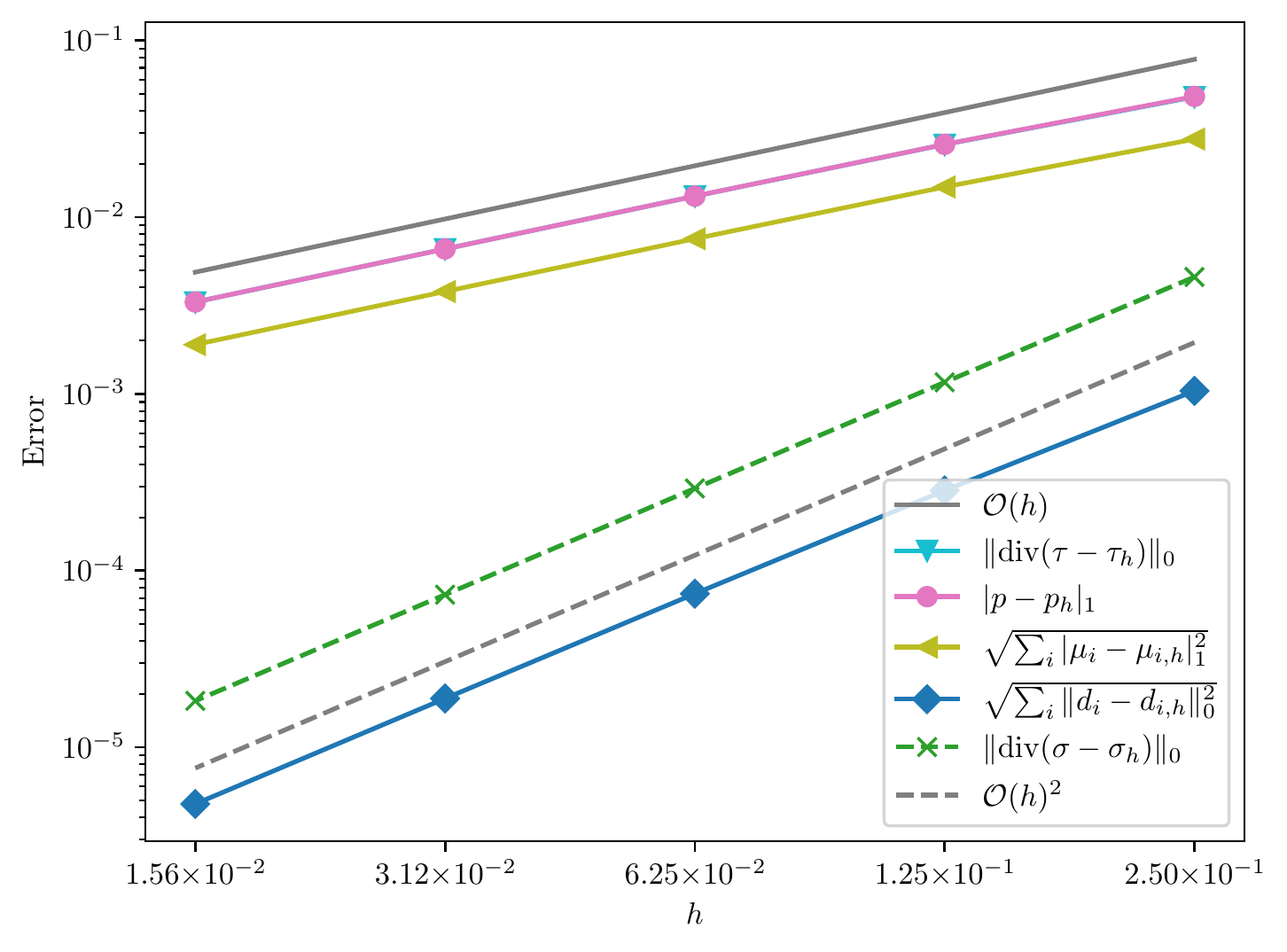}
    \end{subfigure}
    \caption{Higher-order convergence in $L^2$ of the divergence of the full Cauchy stress, and driving forces, for two finite element families:~\cref{eq:canonical-FE-family} (left),
    and~\cref{eq:second-finite-element} (right).
    }\label{fig:theta-conv}
\end{figure}
Remarkably, we observe in~\Cref{fig:theta-conv} that not only do the components $\nabla\mu_i, \nabla p, \Div\tau$ converge, but in fact there is convergence of the nonlinear diffusion driving forces and of the divergence of the full Cauchy stress, and at a rate one order higher than these individual components; this suggests that, rather than being a mathematical artefact of our formulation, the conditions defining the $\Theta^k$ space capture the underlying thermodynamic quantities of interest. This also provides circumstantial evidence towards the physical relevance of our nonlinear formulation in~\Cref{defn:weak-SOSM}.

\subsection{Benzene-cyclohexane mixture}

Cyclohexane (\ch{C6H6}) is important in the petrochemical industry as it is used to synthesize a variety of products, such as nylon. 
It is primarily produced through the dehydrogenation of benzene (\ch{C6H12}), resulting in a benzene-cyclohexane mixture. Separation of cyclohexane from this mixture is difficult due to  their similar vaporization temperatures~\citep{Villaluenga2000}. 
Since liquid mixtures of these components are important in the chemical industry, most of the required thermodynamic and dynamical property data are readily available in the literature. Because it provides a tractable non-ideal example for which a complete set of material properties is known, we consider here a microfluidic chamber in which Stokes flows of benzene and cyclohexane mix.

The required transport parameters (measured at $298.15 \: \text{K}$) may be found in~\citet{Guevara-Carrion2016}. 
e observe from these data that the Stefan--Maxwell diffusivity and the shear viscosity are both approximately constant with respect to composition and pressure, and will be approximated as $\mathscr{D}_{12} = 2.1 \times 10^{-9}~\text{m}^{2}/\text{s}$ and $6 \times 10^{-4}~\text{Pa} \cdot \text{s}$, respectively. Lacking accurate data for the bulk viscosities of either benzene or cyclohexane, we set them to be essentially zero, choosing $\zeta = 10^{-7}~\text{Pa} \cdot \text{s}$. (Numerical experiments confirmed that a value of this order has no discernible impact on the output of the simulation.) The molar masses used in the simulation are $0.078~\text{kg} \cdot \text{mol}^{-1}$ for benzene and $0.084~{\text{kg} \cdot \text{mol}^{-1}}$ for cyclohexane.
The ambient pressure was taken as $p^\ominus = 10^5~\text{Pa}$.

Although benzene and cyclohexane are fully miscible, they form a non-ideal solution. Information relating the chemical potentials to the mole fractions is therefore required. This is accomplished using a Margules model~\citep{Green2007} for activity coefficients, the parameters of which were reported by~\citet{Tasic1978}. This well-known model parameterizes activity coefficients in terms of a minimal set of functions which maintain thermodynamic rigour. 

To aid convergence, we use an under-relaxation scheme with respect to the concentrations in our nonlinear solver, with a relaxation parameter of $0.1$. That is, we update the concentration as $c^{*,k+1}_{i}$ where 
\begin{equation}
    c_{i}^{*,k+1} = 0.9c_{i}^{k} + 0.1c_{i}^{k+1}.
\end{equation}
This is necessary due to stiffness of the problem, which owes to the fact that the mixtures are essentially fully separated at the inlets to the apparatus.

To calculate the total concentration of the mixture we use
\begin{equation}\label{eq:benzene-cyclohexane-eqn-of-state}
    c_{\text{T}} = \frac{c_{\ch{C6H6}}^{\text{ref}} c_{\ch{C6H12}}^{\text{ref}}}{ x_{\ch{C6H6}}c_{\ch{C6H12}}^{\text{ref}} +x_{\ch{C6H12}} c_{\ch{C6H6}}^{\text{ref}}},
\end{equation}
where $c_{-}^{\text{ref}}$ denotes the concentration (inverse molar volume) of the pure species at $10^{5} \: \text{Pa}$ and $298.15 \: \text{K}$, approximately $9.20 \:\: \text{mol} \:\text{L}^{-1}$ and $11.23 \:\: \text{mol} \: \text{L}^{-1}$ for benzene and cyclohexane, respectively. \Cref{eq:benzene-cyclohexane-eqn-of-state} is derived from~\cref{eq:volumetric-eqn-of-state} under the assumption that the partial molar volumes of the two components are independent of the solution's composition. 

We consider a two-dimensional pipe configuration where two inlets converge into a single outlet. At the top inlet, pure benzene enters and at the bottom pure cyclohexane, at speeds $v_{\ch{C6H12}}^{\text{ref}}$ and $v_{\ch{C6H6}}^{\text{ref}}$, respectively. Rather than symmetrize these speeds, superior mixing results are obtained by symmetrizing the molar fluxes at the inlets. In other words, we impose the condition
\begin{equation}\label{eq:molar-symmetry}
    c_{\ch{C6H6}}^{\text{ref}} v_{\ch{C6H6}}^{\text{ref}} =  c_{\ch{C6H12}}^{\text{ref}} v_{\ch{C6H12}}^{\text{ref}}.
\end{equation}
We will specify that $v_{\ch{C6H12}}$ enters at a speed of $4 \mu \text{m} \: \text{s}^{-1}$. 
This prescribes an inlet speed for benzene of approximately $3.28 \mu \text{m} \: \text{s}^{-1}$.   
A parabolic profile across each inlet and outlet is imposed, as this is consistent with the no-slip condition and the characteristics of a plane Poiseuille flow.

\begin{figure}
    \centering
    \includegraphics[width=.7\linewidth]{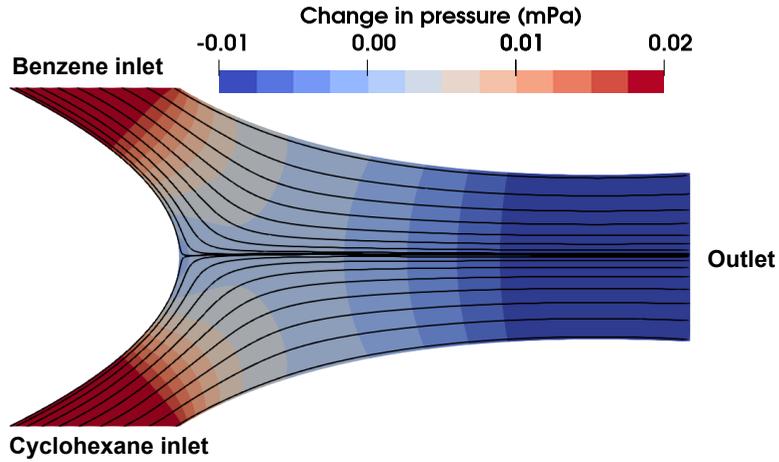}
    \caption{Plot of change in pressure in the mixing chamber, with streamlines computed from the mass-average velocity.}\label{fig:Pressure}
\end{figure}

\begin{figure}
    \centering
    \includegraphics[width=1\linewidth]{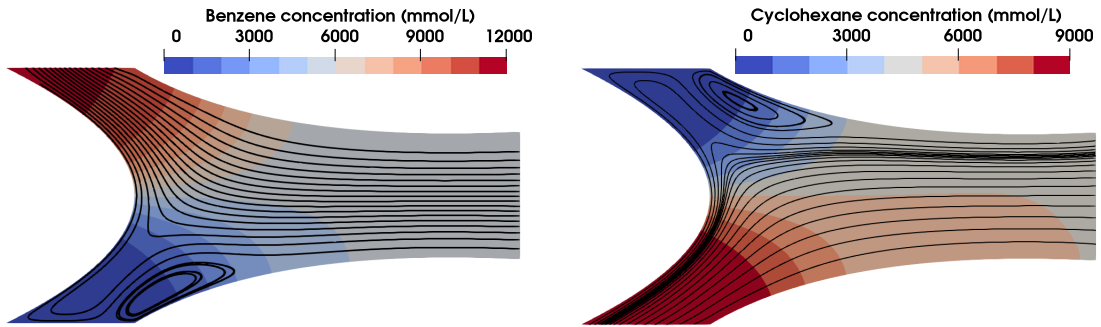}
    \caption{Concentrations of benzene (left) and cyclohexane (right), with streamlines computed from their velocities.}\label{fig:Species}
\end{figure}
Results for the fields computed by the simulation are visualized in Figures~\ref{fig:Pressure} and~\ref{fig:Species}. We observe that the pressure profile is smooth, despite the nonconvex domain. 
We also note that, although the mass-average velocity exhibits rather simple predictable behaviour, the flow fields of the individual species are significantly more complex, and that these three flow profiles are cleanly distinguished. We see that both species develop convective rolls---behaviour markedly different from the convective velocity.

\subsection{Code availability}

For reproducibility, the exact software versions used to generate the numerical results in this paper are archived on Zenodo~\citep{zenodo/Firedrake-20220824.0}; our implementation employs a nondimensionalization of the SOSM system. 
The code, and scripts for the associated plots are available at~\url{https://bitbucket.org/FAznaran/sosm-numerics/}. 

\section{Conclusions and outlook}

We have proposed a formulation and numerical method for the Stokes--Onsager--Stefan--Maxwell equations of multicomponent flow, proving continuous and discrete inf-sup conditions for a linearization of the system with saddle point structure. We obtained error estimates in a norm corresponding to a space requiring square-integrable diffusion driving forces and total stress divergence, and verified these with numerical experiments.

This work represents a first step towards the simulation of non-ideal mixtures; further physical extensions will be required for realistic applications in chemical engineering.
Of particular interest is the analysis of the transient problem, the incorporation of thermal effects based on the framework proposed in~\cite{Van-Brunt2022}, the weak enforcement of symmetry for the viscous stress tensor to ease the extension of the method to three dimensions~\citep[Section 9.2]{Boffi2013}, the consideration of the case of vanishing bulk viscosity as encountered in dilute monatomic gases~\citep{Hirschfelder1954}, and the incorporation of advection into the formulation.

Rigorous investigation into a notion of weak solution more refined than~\Cref{defn:weak-SOSM} incorporating (for example) integrability of thermodynamic pressure gradients, and weak form of linearization, would also be of significant interest. We also remark that a proof of convergence of the Picard iteration could be used to prove the existence of a solution tuple for the infinite-dimensional nonlinear SOSM system.

\section*{Acknowledgements}

This work was supported by the Engineering and Physical Sciences Research Council Centre for Doctoral Training in Partial Differential Equations: Analysis and Applications (EP/L015811/1), the Engineering and Physical Sciences Research Council (grants EP/R029423/1 and EP/W026163/1), The MathWorks, Inc., the Clarendon fund scholarship, and the Faraday institution SOLBAT project and Multiscale Modelling projects (subawards FIRG007 and FIRG003 under grant EP/P003532/1).
The authors are 
also
grateful to J.~M\'alek and M.~P.~Juniper for useful comments.

\bibliographystyle{abbrvnat}
\bibliography{bib}

\end{document}